\documentclass[12pt]{article}
\usepackage[utf8]{inputenc}
\usepackage{amsthm}
\usepackage{amssymb, amsfonts, amsmath, mathtools}
\usepackage[margin=1in]{geometry}
\usepackage{enumitem}
\usepackage{ifthen}
\usepackage[hidelinks]{hyperref}
\usepackage{theoremref}
\usepackage[usenames,dvipsnames]{xcolor}
\usepackage{tikz}
\usepackage{color}
\usepackage{url}
\usepackage{float}
\usepackage[skip=0pt]{caption}
\usepackage{diagbox}
\usepackage{setspace}
\usepackage{graphicx}
\graphicspath{ {./images/} }
\usepackage{thm-restate}
\usepackage{comment}
\usepackage{subfiles}

\allowdisplaybreaks
\usetikzlibrary{decorations.markings}

\def\p{{\mathbf p}}
\def\q{{\mathbf q}}
\def\z{{\mathbf z}}
\def\Z{{\mathbb Z}}

\def\R{\mathbb {R}}
\def\wsat{{\text{wsat}}}

\def\vspan{{\text{span}}}

\makeatletter
\newcommand{\thickhline}{%
    \noalign {\ifnum 0=`}\fi \hrule height 1pt
    \futurelet \reserved@a \@xhline
}
\newcolumntype{"}{@{\hskip\tabcolsep\vrule width 1pt\hskip\tabcolsep}}
\makeatother

\newtheoremstyle{mystyle}
{1.2em} 
{0.5em} 
    {\itshape} 
    {} 
    {\bfseries} 
    {}
    {.5em} 
    {} 
    
\newtheoremstyle{mystyle2}
{1.2em} 
{0.5em} 
    {} 
    {} 
    {\bfseries} 
    {}
    {.5em} 
    {} 
\theoremstyle{mystyle}
\newtheorem{thm}{Theorem}[section]
\theoremstyle{mystyle}
\newtheorem{prop}[thm]{Proposition}
\theoremstyle{mystyle}
\newtheorem{lemma}[thm]{Lemma}
\theoremstyle{mystyle}
\newtheorem{qn}[thm]{Question}
\theoremstyle{mystyle}

\theoremstyle{mystyle}

\theoremstyle{mystyle}
\newtheorem{cor}[thm]{Corollary}
\theoremstyle{mystyle2}
\newtheorem{defn}[thm]{Definition}
\theoremstyle{mystyle2}

\theoremstyle{mystyle2}
\newtheorem{claim}[thm]{Claim}
\theoremstyle{mystyle}

\newlist{mycases}{enumerate}{1}
\setlist[mycases,1]{label=Case~\arabic*:,leftmargin=\labelwidth,align=left, labelindent=0pt, listparindent=\parindent, labelwidth=0pt, itemindent=!, itemsep=3mm}

\newlist{mysubcases}{enumerate}{1}
\setlist[mysubcases,1]{label=Case~\Alph*:,leftmargin=\labelwidth,align=left, labelindent=0pt, listparindent=\parindent, labelwidth=0pt, itemindent=!, itemsep=3mm}

\usetikzlibrary{shapes.geometric}
\definecolor{lgray}{gray}{0.95}
\definecolor{mgray}{gray}{0.40}
\usetikzlibrary{fit,positioning,calc}
\tikzstyle{std}=[ circle, draw=black,fill=black, inner sep=0pt, minimum size=2mm]
\tikzstyle{blank}=[ circle, draw=none,fill=none, inner sep=0pt, minimum size=2mm]
\tikzstyle{bred}=[circle, draw=black,fill=red,thick,  inner sep=2pt, minimum size=2mm]
\tikzstyle{bgreen}=[ circle, draw=black,fill=green,thick,  inner sep=2pt, minimum size=2.5mm]
\tikzstyle{sqRed}=[rectangle, draw=black,fill=red,thick,  inner sep=2pt, minimum size=2.5mm]
\tikzstyle{trEdge}=[color=black, line width = 1pt, style=dotted]
\tikzstyle{circ}=[circle, draw=black,fill=white, inner sep=2pt, minimum size=22.5mm]
\tikzstyle{circl}=[circle, draw=black,fill=white, inner sep=2pt, minimum size=15mm]
\tikzstyle{gcirc}=[circle, draw=gray,fill=white, inner sep=2pt, minimum size=15mm]

\tikzstyle{eblue}=[color=blue, line width = 1pt]
\tikzstyle{ered}=[color=red, line width = 1pt]
\tikzstyle{egreen}=[color=green, line width = 1pt]
\tikzstyle{eorange}=[color=orange, line width = 1pt]
\tikzstyle{eyellow}=[color=yellow, line width = 1pt]
\tikzstyle{epurple}=[color=violet, line width = 1pt]
\tikzstyle{ecyan}=[color=cyan, line width = 1pt]
\tikzstyle{eblack}=[color=black, line width = 1pt]
\tikzstyle{epink}=[color=magenta, line width = 1pt]
\tikzstyle{edgreen}=[color=cadmiumgreen, line width = 1pt]

\begin{document}

\title{\textbf{Using polynomials to find lower bounds for $r$-bond bootstrap percolation}}
\date{\today}
\author{Natasha Morrison \footnotemark[1]\thanks{Department of Mathematics and Statistics, University of Victoria.} \thanks{Research supported by NSERC Discovery Grant RGPIN-2021-02511 and NSERC Early Career Supplement DGECR-2021-00047. Email: \href{mailto:nmorrison@uvic.ca}{nmorrison@uvic.ca}.} 
 \and Shannon Ogden\footnotemark[1] \thanks{Supported by NSERC CGS M. Email: \href{mailto:sogden@uvic.ca}{sogden@uvic.ca}.}}
\maketitle

\begin{abstract}
    The $r$\emph{-bond bootstrap percolation} process on a graph $G$ begins with a set $S$ of \emph{infected} edges of $G$ (all other edges are \emph{healthy}). At each step, a healthy edge becomes infected if at least one of its endpoints is incident with at least $r$ infected edges (and it remains infected). If $S$ eventually infects all of $E(G)$, we say $S$ \emph{percolates}. In this paper we provide recursive formulae for the minimum size of percolating sets in several large families of graphs. We utilise an algebraic method introduced by Hambardzumyan, Hatami, and Qian, and substantially extend and generalise their work.
\end{abstract}

\section{Introduction}
\label{Sec:wsat poly intro}

Bootstrap processes are well-studied examples of cellular automata in graphs. One of the most commonly studied models is the $r$-neighbour bootstrap process, which was originally introduced by physicists to model the dynamics of ferromagnetism \cite{CLR}. Since then, both the $r$-neighbour process and variants have found many exciting applications in the study of physical phenomena (see, e.g.~\cite{AL}), and other fields as diverse as virology, neurology, psychology, and finance.

In this paper, we study the $r$\emph{-bond bootstrap percolation} process, introduced in~\cite{LZ},  on a graph $G$. The process begins with a set of initially \emph{infected} edges of $G$, and we consider all other edges to be \emph{healthy}. At each subsequent step, a healthy edge becomes infected if at least one of its endpoints is incident with at least $r$ infected edges. Once an edge is infected, it remains infected indefinitely. If a set of initially infected edges will eventually infect all of $E(G)$, we say this set \emph{percolates}, and refer to it as an $r$\emph{-percolating set} of $G$. The $r$-bond boostrap percolation process is a particular instance of a very well-studied process called \emph{weak saturation}, which was introduced by Bollob\'{a}s~\cite{Bela} in 1968. See the survey~\cite{survey} for a summary of what is known about this problem. 

Define $m_e(G,r)$ to be the minimum number of edges in an $r$-percolating set in a graph $G$.
Recently, Hambardzumyan, Hatami, and Qian~\cite{HHQ} introduced a new polynomial method (see \thref{HHQ ineq}, below) for determining lower bounds on $m_e(G,r)$. They applied this method to provide bounds when $G = F\square H$, where $F\square H$ denotes the \emph{Cartesian product}\footnote{Recall that $F\square H$ is the graph with vertex set $V(F)\times V(H)$, where two vertices $(u_1,v_1)$ and $(u_2,v_2)$ are adjacent if and only if either $u_1=u_2$ and $v_1v_2\in E(H)$, or $v_1=v_2$ and $u_1u_2\in E(F)$.} of two graphs $F$ and $H$. In what follows, we refer to a Cartesian product of two graphs as simply a \emph{product}. For ease of notation, given graphs $G_1,...,G_k$, we write the product $G_1\square\cdots\square G_k$ as $\prod_{i=1}^k G_i$. 

Hambardzumyan, Hatami, and Qian~\cite{HHQ} used their method to determine $m_e(G,r)$ when $G$ is a $d$-dimensional torus, thus resolving a question of Morrison and Noel~\cite{MN18}. They also proved an analogous result, providing a recursive formula for $m_e(G,r)$ when $G$ is a product of paths (see Theorem 9 in \cite{HHQ}), which gave an alternate simpler proof of a result of Morrison and Noel~\cite{MN18}. In this paper, we push the power of this polynomial method further in order to provide recursive formulas for $m_e(G,r)$ for various families of graphs, hence extending and generalising several results of~\cite{HHQ}. 

As some of our results are fairly technical to state, we present them in designated sections below. In particular, we provide a general recursive formula for $m_e(G\square T,r)$ when $T$ is a tree and the minimum degree of $G$ is sufficiently large compared to $r$ (see \thref{trees exact} below). \thref{trees exact} is a substantial generalisation of Theorem 9 of \cite{HHQ}. 

When $T$ is a star, a more careful analysis allows the minimum degree condition on $G$ to be dropped, giving a recursive formula for $m_e(G\square S_k,r)$ (see \thref{Gen Stars exact} below). Recall that $S_k$ is the star with $k$ leaves. A particular consequence of \thref{Gen Stars exact} is the following recursive formula for $m_e(G,r)$, where $G$ is a product of stars. 

\begin{restatable}{thm}{StarsExact}
\thlabel{Stars exact}
Let $k,r,a_1,...,a_k \in \Z^+$. Let $G_i = S_{a_1}\square \cdots \square S_{a_i}$ for each $i\in [k]$, and let $G_0=K_1$. For $0\leq i\leq \Delta(G_{k-1})$, let $d_i$ denote the number of vertices in $G_{k-1}$ of degree $i$. Then
\[m_e(G_k,r) = m_e(G_{k-1},r) + a_k m_e(G_{k-1},r-1) + \sum\limits_{t=1}^{a_k-1} t d_{r-t} + a_k \sum\limits_{t=a_k}^{r} d_{r-t}.\]
\end{restatable}

Let $k\geq \ell\geq 3$. Define $H_{k,\ell}$ to be the \emph{theta graph} obtained by joining two vertices by three internally vertex-disjoint paths of length $1$, $\ell-1$, and $k-1$. We also provide a general recursive formula for $m_e(G\square H_{k,\ell},r)$ when $k\ge \ell\ge 4$ (see \thref{Gen joined cycles exact} below). As a consequence of \thref{Gen joined cycles exact}, we obtain the following.

\begin{restatable}{thm}{JCyclesExact}
\thlabel{joined cycles exact}
Let $t>0$ and $r>1$ be integers. Let $k_1,...,k_t,\ell_1,...,\ell_t\in \Z^+$ such that $k_i\geq\ell_i\geq 4$ for all $i\in [t]$. Let $G_i = H_{k_1,\ell_1}\square \cdots \square H_{k_i,\ell_i}$ for $i\in [t]$, and let $G_0=K_1$. For $0\leq i\leq \Delta(G_{t-1})$, let $d_i$ denote the number of vertices in $G_{t-1}$ of degree $i$. Then
\begin{align*}
    m_e(G_t,r) &= m_e(G_{t-1},r) + (k_t+\ell_t-5) m_e(G_{t-1},r-1) + 2m_e(G_{t-1},r-2) \\ &\hspace{4em} + d_{r-1} + (k_t+\ell_t-3)d_{r-2} + (k_t+\ell_t-1) \sum\limits_{i=0}^{r-3} d_i.
\end{align*}
\end{restatable}

The paper is organised as follows. In Section~\ref{Sec:Lemma} we introduce the method that will be used to provide the lower bounds in all of our theorems, and prove a useful little lemma. In Section~\ref{Sec:Trees} we focus on results for trees. The main result of that section is \thref{trees exact}. Theta graphs are the focus of Section~\ref{Sec:Joined Cycles}, where \thref{Gen joined cycles exact} is proved. In Section~\ref{Sec:cons}, we explore several interesting consequences of Theorems \ref{trees exact} and \ref{Gen joined cycles exact}, in particular the applications to Theorems \ref{Stars exact} and \ref{joined cycles exact}. Some concluding remarks are given in Section~\ref{sec:concl}, where we discuss the limitations of the method, the relationship to another well-used method for providing lower bounds on weak saturation numbers, and possible future directions for work.

\section{Preliminaries}
\label{Sec:Lemma}

In this section, we provide the background to the algebraic technique employed to prove the lower bound in Theorems \ref{trees exact} and \ref{Gen joined cycles exact}. Note that, by convention, the zero polynomial has degree $-1$. Moreover, for ease of notation, given an edge-colouring $c:E(G)\rightarrow \R$ and an edge $e\in E(G)$, we often denote $c(e)$ by $c_e$.

\begin{defn} 
\thlabel{W}
Let $r\geq0$ be an integer, $G=(V,E)$ a graph, and $c:E\rightarrow \R$ a proper edge-colouring of $G$. Define $W^r_{G,c}$ to be the vector space consisting of all vectors $\p=(p_v)_{v\in V}$ with entries in $\R[x]$ such that
\begin{enumerate}[topsep=0pt]
\item $\deg(p_v)\leq \min\{r,\deg(v)\}-1$ for all $v\in V$, and
\item $p_u(c_{uv})=p_v(c_{uv})$ for every edge $uv\in E$.
\end{enumerate} 
\end{defn}

Note that, if $E(G)=\emptyset$, then $\dim(W^r_{G,c})=0$ for all $r\geq 0$. Furthermore, $\dim(W^0_{G,c})=0$ for any graph $G$ and proper edge-colouring $c:E(G)\rightarrow \R$.

\begin{restatable}{thm}{HHQIneq}
    \thlabel{HHQ ineq}
\emph{\cite{HHQ}} Let $c:E\rightarrow \R$ be a proper edge-colouring of a graph $G=(V,E)$, and let $r\geq 0$ be an integer. Then $m_e(G,r)\geq \dim\left(W^r_{G,c}\right)$.
\end{restatable}

We now prove a lemma that will be useful in determining lower bounds on $m_e(G\square H)$ for various graphs $H$. In order to state \thref{zeros lemma}, we first require the following definition.

\begin{defn}
\thlabel{Z}
    Let $c:E\rightarrow \R$ a proper edge-colouring of a graph $G=(V,E)$. Define $Z_c$ to be the set of vectors $\z =(z_v)_{v\in V}\in \R^{|V|}$ where, for each $v\in V$, the entry $z_v$ is chosen from $\{c_{uv}:u\in N_G(v)\}$. Now, let $H$ be a graph, and let $\p=(p_u)_{u\in V(G\square H)}$ be a vector where each entry $p_u$ is a univariate polynomial in $\R[x]$. For $\z\in Z_c$, define $\p(\z)\in \R^{|V(G\square H)|}$ to be the vector where the entry corresponding to a vertex $(u,v)\in V(G\square H)$ is $p_{(u,v)}(z_u)$. Here we say $\p$ is \emph{evaluated} at $\z$.
\end{defn}

\begin{lemma}
\thlabel{zeros lemma}
Let $G$ and $H$ be graphs, and let $c:E(G)\rightarrow \R$ be a proper edge colouring of $G$. Let $\p=(p_{(u,v)})_{(u,v)\in V(G\square H)}$ be a vector where each entry is a univariate polynomial in $\R[x]$. Suppose there exists $(u,v)\in V(G\square H)$ such that $p_{(u,v)}=q \prod_{\alpha\in A} (x-\alpha)$ for
\begin{itemize}[topsep=-1pt,itemsep=0pt]
    \item some non-zero polynomial $q$ with $\deg(q)\leq \deg_G(u)-1$, and
    \item some set $A\subseteq \R$ with $A\cap c(E(G))=\emptyset$.
\end{itemize}
Then there exists $\z\in Z_c$ such that $\p(\z)\neq 0$.
\end{lemma}

\begin{proof}  
Since $A\cap c(E(G))=\emptyset$, by definition $p_{(u,v)}$ has at most $\deg(q)$ zeros in $c(E(G))$. Thus, since $\deg(q)\leq \deg_G(u)-1$, the polynomial $p_{(u,v)}$ cannot evaluate to zero on all $\deg_G(u)$ distinct colours in $\{c_{uw}:w\in N_G(u)\}$. Let $w\in N_G(u)$ such that $p_{(u,v)}(c_w)\neq 0$. Then $\p(\z)\neq 0$ for any vector $\z=(z_x)_{x\in V(G)}\in Z_c$ such that $z_u=c_{uw}$.
\end{proof}

\section{Products of Trees} 
\label{Sec:Trees}
 
The main goal of this section is to prove  a recursive formula for $m_e(G\square T,r)$ when the minimum degree of $G$ is sufficiently large (compared to $r$). Given a graph $G$, we denote by $d^G_t$ the number of vertices in $G$ of degree~$t$. 

\begin{restatable}{thm}{TreesExact}
\thlabel{trees exact}
    Let $G$ be a graph, let $T$ be a tree of order $n\geq 2$, and let $r\in \Z^+$. Suppose that either $T$ is a path or $G$ is a graph with $\delta(G)\geq r-2$. Suppose there exists a proper edge-colouring $c$ of $G$ such that $m_e(G,i)= \dim(W_{G,c}^i)$ for all $i \in \{r-1, r\}$. Then 
    \begin{align*}
        m_e(G\square T,r)&= m_e(G,r) + (n-1)m_e(G,r-1) + d^G_{r-1} + \sum\limits_{t=0}^{r-2}d^G_{t}\left(1+\sum\limits_{i=2}^{\Delta(T)} d_i^T\right).
    \end{align*} 
\end{restatable}

\thref{trees exact} follows immediately from the next two results. As is usually the case in proving bounds on $m_e(G,r)$, the upper bound in \thref{trees upper} is given by a construction, whilst the algebraic method of \thref{HHQ ineq} is used to give \thref{trees lower}.

\begin{restatable}{prop}{treesUpper}
\thlabel{trees upper}
Let $T$ be a tree of order $n\geq 2$. Let $r\in \Z^+$, and let $G$ be a graph. Then
\begin{align*}
    m_e(G\square T,r)&\leq m_e(G,r) + (n-1)m_e(G,r-1) + d^G_{r-1} \\ &\hspace{4em}  + \sum\limits_{t=1}^{r-1} d_{r-1-t}^G\left( 1+ t\sum\limits_{i=t+1}^{\Delta(T)} d_{i}^T + \sum\limits_{i=2}^{t} (i-1) d_{i}^T \right).
\end{align*} 
\end{restatable}

\begin{restatable}{prop}{treesLowerCor}
\thlabel{trees lower}
Let $T$ be a tree of order $n\geq 2$, and let $r\in \Z^+$. Let $c:E\rightarrow \R$ be a proper edge-colouring of a graph $G=(V,E)$. Then 
\begin{align*}
    m_e(G\square T,r)&\geq \dim\left(W^r_{G,c}\right) + (n-1) \dim\left(W^{r-1}_{G,c}\right) +d_{r-1}^G+\sum\limits_{t=0}^{r-2} d^G_{t}\left(1+\sum\limits_{i=2}^{\Delta(T)}d_i^T\right).
\end{align*}
\end{restatable}

Finally, we show that the condition in \thref{trees exact} requiring either $T$ to be a path or $\delta(G)\geq r-2$ is necessary in order for the bounds given in \thref{trees upper,trees lower} to match.

\begin{prop}
\thlabel{lem:trees exact}
    Let $T$ be a tree of order $n\geq 2$, and $r\in \Z^+$. Let $c:E\rightarrow \R$ be a proper edge-colouring of a graph $G=(V,E)$, and suppose $m_e(G,i)=\dim(W^i_{G,c})$ for $r-1\leq i\leq r$. If the upper and lower bounds on $m_e(G\square T)$ given in \thref{trees upper,trees lower} match, then either $\delta(G)\geq r-2$ or $T$ is a path. 
\end{prop}

We will use the following notation throughout. Given a tree $T$, we consider $T$ to be rooted at a leaf, which we label $0$. Let the levels of $T$ with respect to this root be labelled $L_0,L_1,...,L_k$ for some $k\geq 1$, where $L_0$ contains the root of $T$, and $v\in L_d$ for each $v\in V(T)$ such that $d_T(0,v)=d$. Let $l(v)$ denote the level of $v\in V(T)$; that is, $l(v)=d$ if and only if $v\in L_d$. Label the remaining vertices of $T$ with the integers in $[n-1]$ such that $l(i)\leq l(j)$ for all $0\leq i<j\leq n-1$. If $ij\in E(T)$ with $l(i)=l(j)+1$, then we call $j$ the \emph{parent} of $i$ and $i$ a \emph{child} of $j$. If $j$ is the parent of $i$ in $V(T)$, we denote by $T_i$ the component of $T \setminus \{ij\}$ containing the vertex $i$; we consider $T_i$ to be rooted at $i$.

Given a graph $G$, we let $G_0,G_1,...,G_{n-1}$ denote the $n$ copies of $G$ in $G\square T$ corresponding to the $n$ vertices in $T$, where $G_i$ corresponds to the vertex $i\in V(T)$. Define $l(G_i):=l(i)$ for each $i\in V(T)$. For ease of notation, we write $G_i\in L_d$ when $l(G_i)=d$. For each vertex $v\in V(G)$, denote its corresponding vertex in $G_i$ by $v_i$, where $0\leq i\leq n-1$. If $v_iv_j\in E(G\square T)$ with $l(i)=l(j)+1$, then we call $v_j$ the \emph{parent} of $v_i$ and $v_i$ a \emph{child} of $v_j$.     

\begin{proof}[Proof of Proposition~\ref{trees upper}]

First, we consider the case where $\delta(G)\geq r$. We construct an $r$-percolating set $F\subseteq E(G\square T)$ for $G\square T$ as follows: Let $F_0$ be an optimal $r$-percolating set for $G_0$. For each $i\in [n-1]$, pick an optimal $(r-1)$-percolating set $F_i$ on $G_i$. Let $F=\bigcup_{i=0}^{n-1}F_i$. 

We claim that $F$ percolates in $G$, starting with the edges in $G_0$ and then spreading through those in $L_1,...,L_k$, in that order. Indeed, by our choice of $F_0$, after running the $r$-bond bootstrap percolation process on $G_0$, all edges in $G_0$ will be infected. Since $\delta(G)\geq r$, all edges between $G_0$ and $G_1$ will then become infected. Now, since each vertex in $G_1$ has an infected edge coming from $G_0$, this, together with $F_1$, will infect the edges of $G_1$. Again, since $\delta(G)\geq r$, we can then infect all edges between $G_1$ and each $G_i\in L_2$. Now, for each $i\in L_2$, every vertex $v_i\in V(G_i)$ has an infected edge coming from $G_1$, which, together with $F_i$, will infect all edges in $G_i$. Hence we can infect the edges in each $G_i\in L_2$. Continuing in this manner, we can infect $L_1,...,L_k$ in order. Hence $F$ percolates in $G\square T$, as required.

It remains to consider the case where $\delta(G)<r$. Let $F'$ be obtained from $F$ by adding the following edges: 
\begin{enumerate}[nosep]
    \item[(i)] For each $v\in V(G)$ with $\deg_G(v)=r-1$, add the edge $v_0v_1$ to the set $F$. 
    
    Note that adding this edge will guarantee that, once the edges in $G_0$ are infected, the infection will spread to those in $G_1$, and we can then continue as above. This adds an additional $d^G_{r-1}$ edges to the set $F$.
    \item[(ii)] Fix $t\in [r-1]$. For each $v\in V(G)$ such that $\deg_G(v)=r-1-t$, add the following edges to the set $F$: 
    \begin{itemize}[nosep]
        \item Add the edge $v_0v_1$.
        \item For each $a\in V(T)$ with $\deg_T(a)>t$, add $t$ edges from $v_a$ to its children. This adds an additional $t\sum_{i=t+1}^{\Delta(T)} d_{i}^T$ edges to $F$.
        \item For each $i\in \{2,...,t\}$, and each vertex $a\in V(T)$ with $\deg_T(a)=i$, add all $i-1$ edges from $v_a$ to its children. This adds an additional $\sum_{i=2}^{t} (i-1) d_{i}^T$ edges to $F$.
    \end{itemize}

    Note that adding these edges will guarantee that, for each $0\leq a\leq n-1$, once the edges in $G_a$ are infected, the infection will spread to the remaining edges $v_av_b$ where $l(b)=l(a)+1$.
\end{enumerate} 
Now, in total, we added
\[d^G_{r-1} + \sum\limits_{t=1}^{r-1} d_{r-1-t}^G\left( 1+ t\sum\limits_{i=t+1}^{\Delta(T)} d_{i}^T + \sum\limits_{i=2}^{t} (i-1) d_{i}^T \right)\]
edges to the $m_e(G,r) + (n-1)m_e(G,r-1)$ edges initially in $F$. Therefore, since the resulting set $F'$ percolates in $G$, we have established the desired upper bound on $m_e(G\square T,r)$.
\end{proof}

We now prove the lower bound. 

\begin{proof}[Proof of \thref{trees lower}]

It suffices to show there exists a proper edge colouring $c^{\prime}$ of $G\square T$ such that 
\begin{equation}\label{eq:lb-tree}
\dim\left(W^r_{G\square T,c^\prime}\right) \geq \dim\left(W^r_{G,c}\right) + (n-1) \dim\left(W^{r-1}_{G,c}\right) +d_{r-1}^G+\sum\limits_{t=0}^{r-2} d^G_{t}\left(1+\sum\limits_{i=2}^{\Delta(T)}d_i^T\right) := N,
\end{equation}
as by \thref{HHQ ineq} we have $m_e(G\square T,r)\geq \dim(W^r_{G\square T,c'})$.

Let $\{\alpha_i:i\in [n-1]\}$ be distinct real numbers that do not belong to $c(E)$. Let $c^{\prime}$ be the proper edge-colouring of $G\square T$ defined as follows: Let $c'$ be consistent with $c$ on each of $G_0,G_1,...,G_{n-1}$. For each $v\in V(G)$, let $c^{\prime}(v_iv_j)=\alpha_i$ for all $i,j\in V(T)$ where $j$ is the parent of $i$ in $T$.

To prove \eqref{eq:lb-tree}, we find a set $Y$ of $N$ linearly independent vectors in $W^r_{G\square T,c^\prime}$. We begin by finding a set $A$ of $\dim\left(W^r_{G,c}\right) + (n-1) \dim\left(W^{r-1}_{G,c}\right)$ linearly independent vectors in $W^r_{G\square T,c^\prime}$.

Let $B^{(r)}$ be a basis for $W^r_{G,c}$. For each vector $\q=(q_v)_{v\in V(G)}$ in $B^{(r)}$, define the vector $\p_\q^{0}=(p_u)_{u\in V(G\square T)}$ so that $p_{v_i}=q_v$ for all $0\leq i\leq n-1$ and $v\in V(G)$. Trivially, the two conditions in \thref{W} are satisfied, and thus $\p_\q^{0}\in W^r_{G\square T,c^{\prime}}$. Let $A_0:=\{\p_\q^{0}:q\in B^{(r)}\}$. Note that the restriction of any vector $\p_\q^{0}\in A_0$ to $G_0$ equals $\q$. Therefore, since $B^{(r)}$ is a basis, the vectors in $A_0$ are linearly independent. 

Let $B^{(r-1)}$ be a basis for $W^{r-1}_{G,c}$. Fix $\ell\in [n-1]$. For each $\q=(q_v)_{v\in V(G)}$ in $B^{(r-1)}$, define the vector $\p_\q^{\ell}=(p_u)_{u\in V(G\square T)}$ such that, for each $v\in V(G)$, we have $p_{v_i}=(x-\alpha_\ell)q_v$ for all $i\in V(T_\ell)$, and $p_{v_i}\equiv 0$ for all other $i\in V(T)$. For each $v\in V(G)$, note that $\q\in W^{r-1}_{G,c}$ implies $\deg(q_v)\leq \min\{r-1,\deg_G(v)\} -1$. Therefore, for each $i\in V(T_\ell)$, we have  $\deg(p_{v_i}) = \deg(q_v) + 1 \le \min\{r,\deg(v_i)\}-1.$ Moreover, $\p_{v_\ell}(\alpha_\ell)=0=\p_{v_j}(\alpha_\ell)$, where $j$ is the parent of $\ell$ in $T$. Since $p_u(c'_{uv})=p_v(c'_{uv})$ for all other edges $uv\in E(G\square T)$, the conditions in \thref{W} are satisfied, and thus $\p_\q^{\ell}\in W^{r}_{G\square T,c^\prime}$. 

For each $\ell\in [n-1]$, define $A_\ell:=\{\p_\q^{\ell}:\q\in B^{(r-1)}\}$. Note that, for each $\ell\in [n-1]$, the restriction of any vector $\p_\q^{\ell}\in A_\ell$ to $G_\ell$ equals $(x-\alpha_\ell)\q$. Therefore, since $x-\alpha_\ell\not\equiv 0$ and $B^{(r-1)}$ is a basis, the vectors in $A_\ell$ are linearly independent. 

Define $A:=\bigcup_{i=0}^{n-1} A_i$. Note that $|A_0|=\dim(W^r_{G,c})$, and $|A_\ell|=\dim(W^{r-1}_{G,c})$ for each $\ell\in [n-1]$. Therefore, $|A|=\dim(W^r_{G,c})+ (n-1)\dim(W^{r-1}_{G,c})$.

\begin{claim}
\thlabel{claim:trees A ind}
    The vectors in $A$ are linearly independent. 
\end{claim}

\noindent\emph{Proof of \thref{claim:trees A ind}.} Suppose for a contradiction that there exists a linear combination $S=\sum_{\p\in A} \beta_\p\p = 0$, where $\beta_\p\neq0$ for some $\p\in A$. Write $S=\sum_{j=0}^{n-1}S_j$ where $S_j:=\sum_{\p\in A_j}\beta_\p\p$ for each $j\in V(T)$. Let $i$ be the smallest index for which there exists $\p\in A_i$ with $\beta_{\p}\neq 0$. For each $\p\in \bigcup_{j=i+1}^{n-1 }A_j$, by definition $p_{v_i}\equiv 0$ for all $v\in V(G)$. By our choice of $i$, $\beta_\p= 0$ for all $\p\in \bigcup_{j=0}^{i-1}A_j$. Thus the restriction of the sum $S$ to the coordinates corresponding to vertices in $G_i$ is equal to the restriction of $S_i$ to $G_i$. However, since the vectors in $A_i$ are linearly independent, this contradicts our assumption that $\beta_\p\neq 0$ for some $\p\in A_i$.$_{(\square)}$

To achieve the desired lower bound on $\dim(W^r_{G\square T,c^\prime})$, we must extend $A$ to include an additional $d_{r-1}^G+\sum_{t=0}^{r-2} d^G_{t}(1+\sum_{i=2}^{\Delta(T)}d_i^T)$ vectors, all of which are linearly independent. To this end, consider the set of vectors $Z_c\in \R^{|V(G)|}$ defined in \thref{Z}.

\begin{claim} 
\thlabel{trees claim}
For each non-zero $\p\in \vspan(A)$, there exists $\z\in Z_c$ such that $\p(\z)\neq 0$.
\end{claim}

\noindent\emph{Proof of Claim:} Let $\p= (p_u)_{u\in V(G\square T)}$ be a non-zero vector in $\vspan(A)$. Let $j$ be the smallest index for which there exists some $v_j\in V(G_j)$ with $p_{v_j}\not\equiv 0$. If $j=0$, then by definition of the vectors in $A$, we can write $p_{v_0}=q_v$ for some vector $\q=(q_u)_{u\in V(G)}$ in $\vspan(B^{(r)})=W^{r}_{G,c}$. Otherwise, if $j\in [n-1]$, then we can write $p_{v_j}=q_v(x-\alpha_j)$ for some vector $\q=(q_u)_{u\in V(G)}$ in $\vspan(B^{(r-1)})=W^{r-1}_{G,c}$. In either case, by \thref{W}, $\deg(q_v)\leq \deg_G(v)-1$. Thus, since the colours in $\{\alpha_i:i\in [n-1]\}$ were chosen to be distinct from those in $c(E)$, by \thref{zeros lemma} there exists $\z\in Z_c$ such that $\p(\z)\neq 0$.~$_{(\square)}$

Therefore, in order to finish the proof of \eqref{eq:lb-tree}, it suffices to find a set $X$ containing $d_{r-1}^G+\sum_{t=0}^{r-2} d^G_{t}(1+\sum_{i=2}^{\Delta(T)}d_i^T)$ linearly independent vectors in $W^r_{G\square T, c^\prime}$ which evaluate to zero on all of $Z_c$. This, together with \thref{trees claim}, will ensure these vectors are independent from $A$. 

First, for each $v\in V(G)$ with $\deg_G(v)\leq r-1$, define the vector $\p^0_v=(p_u)_{u\in V(G\square T)}$ so that $p_{v_i}=\prod_{u\in N_G(v)}(x-c_{uv})$ for all $i\in V(T)$, and $p_u\equiv 0$ for all other $u\in V(G\square T)$. Note that $p_u(c'_{uv})=p_v(c'_{uv})$ for all edges $uv\in E(G\square T)$. Furthermore, for all $i\in V(T)$, since $\deg_G(v)\leq r-1$ and $\deg_G(v) \leq \deg(v_i)-1$, we have $\deg(p_{v_i})=\deg_G(v)\leq \min\{r,\deg(v_i)\}-1.$ Hence the conditions in \thref{W} are satisfied, and thus $\p^0_v\in W^r_{G\square T,c'}$. Moreover, $\p^0_v$ evaluates to $0$ on all of $Z_c$. Let $X_v:=\{\p^0_v\}$ for each $v\in V(G)$ with $\deg_G(v)\leq r-1$.

Now, let $v\in V(G)$ with $\deg_G(v)\leq r-2$. For each $i\in V(T)$ with $\deg_T(i)\geq 2$, define the vector $\p_v^i=(p_u)_{u\in V(G\square T)}$ as follows: Let $p_{v_i}=(x-\alpha_i)\prod_{u\in N_G(v)}(x-c_{uv})$. For each child $j$ of $i$ in $T$, let
$p_{v_\ell}=(\alpha_j-\alpha_i)\prod_{u\in N_G(v)}(x-c_{uv})$ for all $\ell\in V(T_j)$.
Finally, let $p_u\equiv 0$ for all other $u\in V(G\square T)$. See Figure \ref{fig:trees low extra} for an illustration of this construction when $i=2$.

Now, $\deg_T(i)\geq 2$ implies $\deg_G(v)\leq \deg(v_i)-2$. Therefore, since $\deg_G(v)\leq r-2$, we have that
$\deg(p_{v_i})=\deg_G(v)+1\leq \min\{r,\deg(v_i)\}-1$. 
In addition, for each child $j$ of $i$ in $T$, and each $\ell\in V(T_j)$, since $\deg_G(v)\leq \deg(v_\ell)-1$ and $\deg_G(v)\leq r-2$, we have
$\deg(p_{v_\ell})=\deg_G(v)\leq  \min\{r,\deg(v_\ell)\}-1$.
Now, $p_{v_i}(\alpha_i)=0=p_{v_a}(\alpha_i)$,
where $a$ is the parent of $i$ in $T$. Moreover, for every child $j$ of $i$ in $T$,
\vspace{-2.75mm}
\[p_{v_i}(\alpha_j)=(\alpha_j-\alpha_i)\prod\limits_{u\in N_G(v)}(\alpha_j-c_{uv})=p_{v_j}(\alpha_j).\]

\begin{figure}[H]
\centering
\begin{tikzpicture}[scale=1.2]

    \draw[gray]  (0,5) ellipse (1.5 and 0.9);
    \draw[gray]  (0,2.5) ellipse (1.5 and 0.9);
    \draw[gray]  (0,0) ellipse (1.5 and 0.9);
    \draw[gray]  (-3,-2.5) ellipse (1.5 and 0.9);
    \draw[gray]  (3,-2.5) ellipse (1.5 and 0.9);
    \draw[gray]  (-3,-5) ellipse (1.5 and 0.9);
    \draw[gray]  (3,-5) ellipse (1.5 and 0.9);

    \node [] (G0) at (0.3,4.5)[]{\textcolor{gray}{$G_0$}};
    \node [] (G1) at (0.3,2)[]{\textcolor{gray}{$G_1$}};
    \node [] (G2) at (0,-0.5)[]{\textcolor{gray}{$G_2$}};
    \node [] (G3) at (-2.7,-3)[]{\textcolor{gray}{$G_{3}$}};
    \node [] (G4) at (2.7,-3)[]{\textcolor{gray}{$G_{4}$}};
    \node [] (G5) at (-2.7,-5.5)[]{\textcolor{gray}{$G_{5}$}};
    \node [] (G6) at (2.7,-5.5)[]{\textcolor{gray}{$G_{6}$}};

    \node [] (a1) at (0.25,3.8)[]{$\alpha_1$};
    \node [] (a2) at (0.25,1.3)[]{$\alpha_{2}$};
    \node [] (a3) at (-1.5,-1.3)[]{$\alpha_{3}$};
    \node [] (a4) at (1.5,-1.3)[]{$\alpha_{4}$};
    \node [] (a5) at (-2.7,-3.7)[]{$\alpha_{5}$};
    \node [] (a6) at (2.7,-3.7)[]{$\alpha_{6}$};

    \node [] (p0) at (-0.7,5.25)[]{$p_{v_0}=0$};
    \node [] (p1) at (-0.7,2.75)[]{$p_{v_1}=0$};
    \node [] (p2) at (3.5,0.9)[]{$p_{v_2}=(x-\alpha_2)\prod\limits_{u\in N_G(v)}(x-c_{uv})$};
    \node [] (p3) at (-4.1,-0.9)[]{$p_{v_{3}}=(\alpha_3-\alpha_2)\prod\limits_{u\in N_G(v)}(x-c_{uv})$};
    \node [] (p4) at (4.1,-0.9)[]{$p_{v_{4}}=(\alpha_4-\alpha_2)\prod\limits_{u\in N_G(v)}(x-c_{uv})$};
    \node [] (p5) at (-3.8,-6.7)[]{$p_{v_{5}}=(\alpha_3-\alpha_2)\prod\limits_{u\in N_G(v)}(x-c_{uv})$};
    \node [] (p6) at (3.8,-6.7)[]{$p_{v_{6}}=(\alpha_4-\alpha_2)\prod\limits_{u\in N_G(v)}(x-c_{uv})$};

    \node [std] (0) at (0,5.25)[label=right: $v_0$]{};
    \node [std] (1) at (0,2.75)[label=right: $v_1$]{};
    \node [std] (2) at (0,0.25)[label=right: $v_2$]{};
    \node [std] (3) at (-3,-2.25)[label=right: $v_{3}$]{};
    \node [std] (4) at (3,-2.25)[label=left: $v_{4}$]{};
    \node [std] (5) at (-3,-4.75)[label=right: $v_{5}$]{};
    \node [std] (6) at (3,-4.75)[label=left: $v_{6}$]{};

    \draw [-stealth] (0.9,1) -- (0.1,0.4);
    \draw [-stealth] (1.9,-1) -- (2.95,-2.1);
    \draw [-stealth] (-6.45,-1) -- (-3.15,-2.15);
    \draw [-stealth] (-6.1,-6.3) -- (-3.15,-4.85);
    \draw [-stealth] (1.5,-6.3) -- (2.85,-4.85);
    
    \draw (0)--(1);
    \draw (1)--(2);
    \draw (2)--(3);
    \draw (2)--(4);
    \draw (3)--(5);
    \draw (4)--(6);

\end{tikzpicture}
\caption{The vector $\p_v^2$ defined for $v\in V(G)$ with $\deg_G(v)\leq r-2$.}
\label{fig:trees low extra}
\end{figure}
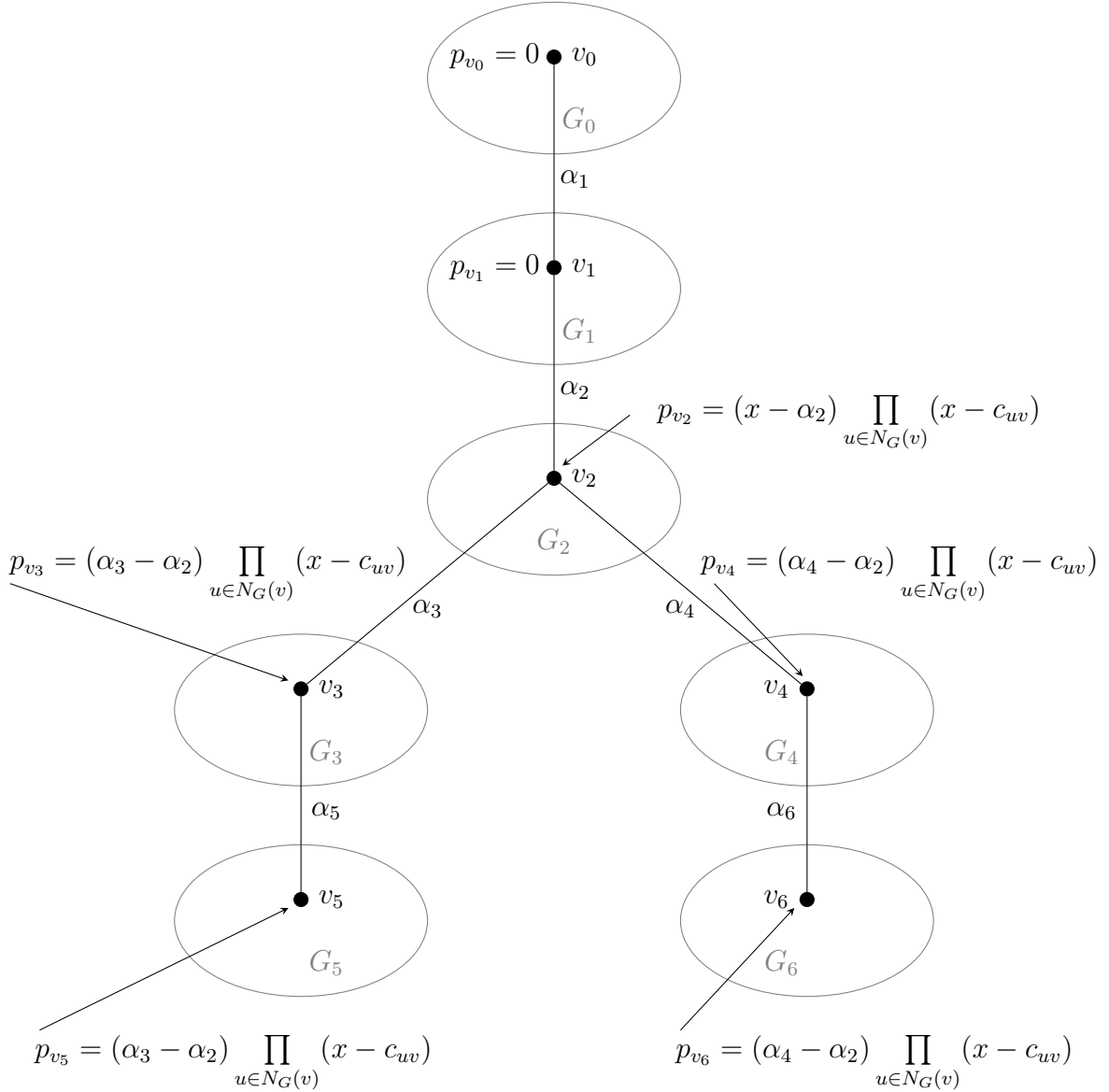

\noindent Since $p_u(c'_{uv})=p_v(c'_{uv})$ for all other $uv\in E(G\square T)$, the conditions in \thref{W} are satisfied, and thus $\p_v^i\in W_{G\square T,c'}^r$. Note that $\p_v^i$ also evaluates to zero on all of $Z_c$. For each $v\in V(G)$ with $\deg_G(v)\leq r-2$, and each $2\leq i\leq \Delta(T)$, add $\p_v^i$ to the set $X_v$.

Let $X:=\bigcup X_v$, where the union is taken over all sets $X_v$ where $\deg_G(v)\leq r-1$. Then 
\[|X|=\sum\limits_{t=0}^{r-1} d_{t}^G + \sum\limits_{t=0}^{r-2} d^G_{t}\sum\limits_{i=2}^{\Delta(T)}d_i^T=d_{r-1}^G+\sum\limits_{t=0}^{r-2} d^G_{t}\left(1+\sum\limits_{i=2}^{\Delta(T)}d_i^T\right).\] 

\begin{claim}
\thlabel{claim:trees X ind}
    The vectors in $X$ are linearly independent.
\end{claim}

\noindent\emph{Proof of \thref{claim:trees X ind}.} Since each vector in $X$ is only non-zero in the entries corresponding to a unique vertex $v\in V(G)$, it suffices to show that the vectors in each $X_v$ are linearly independent. If $\deg_G(v)=r-1$, then $|X_v|=1$, and thus $X_v$ is clearly linearly independent. So we may assume $v\in V(G)$ with $\deg_G(v)=t$, where $t\leq r-2$.

Suppose for a contradiction that $S=\sum_{\p\in X_v}\gamma_{\p} \p=0$ is a non-trivial linear combination. For ease of notation, we write $\gamma_i$ for $\gamma_{\p_v^i}$. Let $i$ be the smallest index such that $\gamma_i\neq 0$. By definition, $p_{v_i}\not\equiv 0$ in $\p^i_{v}$, but $p_{v_i}\equiv 0$ for all $\p^j_{v}$ where $j>i$. Since $\gamma_j=0$ for all $j<i$, restricting the sum $S$ to the coordinate corresponding to $p_{v_i}$ gives $\gamma_i\p^i_v|_{v_i}\equiv 0$. Therefore, since $p_{v_i}\not\equiv 0$ in $\p^i_{v}$, it follows that $\gamma_i=0$, a contradiction. Hence the vectors in $X_v$ are linearly independent.~$_{(\square)}$ 

Let $Y=A\cup X$. Therefore, by \thref{claim:trees A ind,trees claim,claim:trees X ind}, we have found $|Y|=|A|+|X| =N$ linearly independent vectors in $W^r_{G\square T,c'}$, as required.
\end{proof}

It is easy to check that if $T$ is a path or $\delta(G)\geq r-2$ then the bounds in \thref{trees upper,trees lower} match, thus proving \thref{trees exact}. We finish this section by proving that these are the only conditions under which these bounds match. 

\begin{proof}[Proof of \thref{lem:trees exact}]
If the upper and lower bounds on $m_e(G\square T)$ given in \thref{trees upper,trees lower} match, since $m_e(G,i)=\dim(W^i_{G,c})$ for $r-1\leq i\leq r$, 
\[d^G_{r-1} + \sum\limits_{t=1}^{r-1} d_{r-1-t}^G\left( 1+ \sum\limits_{i=t+1}^{\Delta(T)} t d_{i}^T + \sum\limits_{i=2}^{t} (i-1) d_{i}^T \right) = d_{r-1}^G+\sum\limits_{t=0}^{r-2} d^G_{t}\left(1+\sum\limits_{i=2}^{\Delta(T)}d_i^T\right).\]
Cancelling like terms and reindexing the sums on the left hand side, it follows that
\begin{align*} \sum\limits_{t=0}^{r-2} d_{t}^G\left( 1+ \sum\limits_{i=r-t}^{\Delta(T)} (r-1-t) d_{i}^T + \sum\limits_{i=2}^{r-1-t} (i-1) d_{i}^T \right) = \sum\limits_{t=0}^{r-2} d^G_{t}\left(1+\sum\limits_{i=2}^{\Delta(T)}d_i^T\right).  
\end{align*}
Therefore, for each $t\leq r-2$ such that $d_t^G>0$,   
\begin{equation}
  \sum\limits_{i=r-t}^{\Delta(T)} (r-1-t) d_{i}^T + \sum\limits_{i=2}^{r-1-t} (i-1) d_{i}^T = \sum\limits_{i=2}^{\Delta(T)}d_i^T. \label{eqtrees} 
\end{equation}
Fix $t\leq r-2$ such that $d_t^G>0$. From the first sum on the left side of \eqref{eqtrees}, if there exists $r-t\leq i\leq \Delta(T)$ such that $d_i^T>0$, then $r-1-t=1$. Therefore, since $d_{\Delta(T)}^T>0$, either
\begin{enumerate}[topsep=0pt]
    \item [(a)] $t=r-2$, or
    \item [(b)] $\Delta(T)< r-t$.
\end{enumerate}
From the second sum on the left side of \eqref{eqtrees}, if $r-1-t> 2$ then $d_i^T=0$ for all $3\leq i\leq r-1-t$. Therefore, either 
\begin{enumerate}[topsep=0pt]
    \item [(c)] $t\geq r-3$, or 
    \item [(d)] $T$ has no vertices of degree $3\leq i\leq r-1-t$. 
\end{enumerate}

Now, suppose for a contradiction that $\delta(G)<r-2$ and $T$ is not a path. Since $\delta(G)<r-2$, there exists some $t\leq r-3$ such that $d_t^G>0$. Note that condition (a) fails for this value of $t$, and thus condition (b) holds; that is, $\Delta(T)<r-t$. Now, since $T$ is not a path, $\Delta(T)\geq 3$. Hence $3<r-t$. Therefore, condition (c) fails for this value of $t$, and thus condition (d) holds; that is, $T$ has no vertices of degree $3\leq i\leq r-1-t$. However, $T$ contains at least one vertex of degree $\Delta(T)$, contradicting the fact that $3\leq \Delta(T)\leq r-1-t$. Therefore, if the upper and lower bounds on $m_e(G\square T)$ given in \thref{trees upper,trees lower} match, then either $\delta(G)\geq r-2$ or $T$ is a path.
\end{proof}

\subsection{Products of Stars}
\label{Subsec:Stars}

As mentioned in Section \ref{Sec:wsat poly intro}, when $T$ is a star, a more careful analysis allows the minimum degree condition of \thref{trees exact} to be dropped, providing the following recursive formula for $m_e(G\square S_k,r)$, where $S_k$ is a star with $k$ leaves. 

\begin{restatable}{thm}{GenStarsExact}
    \thlabel{Gen Stars exact}
Let $k,r\geq 1$ be integers, and let $G$ be a graph. Suppose there exists a proper edge-colouring $c$ of $G$ such that $m_e(G,i)=\dim(W^{i}_{G,c})$ for each $r-1\leq i\leq r$. Then
\vspace{-2mm}
\[m_e(G\square S_k,r) = m_e(G,r) + k m_e(G,r-1) + \sum\limits_{t=1}^{k-1} t d^G_{r-t} + k \sum\limits_{t=k}^{r} d^G_{r-t}.\]
\end{restatable}

\vspace{-3mm}

As in the proof of \thref{trees exact}, we provide a construction for the upper bound in \thref{stars upper}, and use properties of polynomials to obtain the lower bound in \thref{stars lower}, which, together with \thref{HHQ ineq}, directly implies \thref{Gen Stars exact}. 

\begin{prop}
\thlabel{stars upper}
Let $r,k\geq 1$ be integers, and let $G$ be a graph. Then
\vspace{-2mm}
\[m_e(G\square S_k,r)\leq m_e(G,r) + km_e(G,r-1) + \sum\limits_{t=1}^{k-1} t d^G_{r-t} + k\sum\limits_{t=k}^r d^G_{r-t}.\]
\end{prop}

\begin{prop}
\thlabel{stars lower}
Let $r,k\geq 1$ be integers. Let $c:E\rightarrow \R$ be a proper edge-colouring of a graph $G=(V,E)$. Then there exists a proper edge-colouring $c^{\prime}$ of $G\square S_k$ for which 
\vspace{-2mm}
\[\dim\left(W^r_{G\square S_k,c^\prime}\right) \geq \dim\left(W^r_{G,c}\right) + k \dim\left(W^{r-1}_{G,c}\right) + \sum\limits_{t=1}^{k-1}t d^G_{r-t} + k \sum\limits_{t=k}^r d^G_{r-t}.\]
\end{prop}

As the details of the proofs of \thref{stars upper,stars lower} are very similar in spirit to those of \thref{trees upper,trees lower}, respectively, we omit them here for the sake of brevity. The details can be found in the Appendix, or in Section 3.4 of \cite{thesis}.

\section{Products of Theta Graphs}
\label{Sec:Joined Cycles}

In this section, we establish bounds on $m_e(G\square H,r)$ when $H$ is a \emph{theta graph}.

\begin{defn}
\thlabel{joined cycles}
Let $k\geq \ell\geq 3$. Define $H_{k,\ell}$ to be the \emph{theta graph} obtained by joining two vertices by three internally vertex-disjoint paths of length $1$, $\ell-1$, and $k-1$. Note that $H_{k,\ell}$ contains two induced cycles, one of length $\ell$ and the one of length $k$.

Throughout Section \ref{Sec:Joined Cycles}, we label the vertices of $H_{k,\ell}$ as follows: Label the $\ell$ vertices on the induced cycle of length $\ell$ as $1,...,\ell$ in order so the two vertices of degree $3$ are labelled $1$ and $2$. Label the vertices on the induced cycle of length $k$ with the symbols $1^\prime,...,k^\prime$ in order, so that $1^\prime=1$ and $2^\prime=2$. See Figure \ref{fig:jcycle defn} for an illustration.
\end{defn}

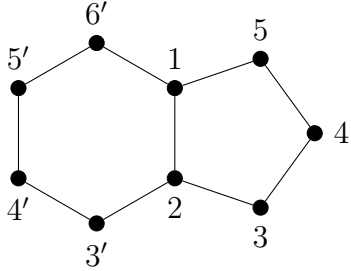
\begin{figure}[H]
\centering
\begin{tikzpicture}[scale=0.6]

    \node [std] (1) at (0,1)[label=above: {$1$}]{};
    \node [std] (2) at (0,-1)[label=below: {$2$}]{};
    \node [std] (3) at (1.9,-1.65)[label=below: $3$]{};
    \node [std] (4) at (3.1,0)[label=right: $4$]{};
    \node [std] (5) at (1.9,1.65)[label=above: $5$]{};
    \node [std] (3') at (-1.732,-2)[label=below: $3'$]{};
    \node [std] (4') at (-3.464,-1)[label=below: $4'$]{};
    \node [std] (5') at (-3.464,1)[label=above: $5'$]{};
    \node [std] (6') at (-1.732,2)[label=above: $6'$]{};

    \draw (1)--(2);
    \draw (2)--(3);
    \draw (3)--(4);
    \draw (4)--(5);
    \draw (5)--(1);
    \draw (2)--(3');
    \draw (3')--(4');
    \draw (4')--(5');
    \draw (5')--(6');
    \draw (6')--(1);
\end{tikzpicture}
\caption{The theta graph $H_{6,5}$}
\label{fig:jcycle defn}
\end{figure}

For ease of notation, let $[k^\prime]=\{1^\prime,...,k^\prime\}$. Given integers $k\geq \ell\geq 3$, define $I=I_{k,\ell}$ to be the set $[\ell]\cup[k^\prime]\setminus\{1^\prime,2^\prime\}$. Hence $V(H_{k,\ell})=I_{k,\ell}$. We impose an ordering $\phi$ on $\Z\cup[k']$ where $i<i'<i+1$ for all $i\in [k]$. Note that $I\subseteq \Z\cup[k']$, and thus $\phi$ is also an ordering on $I$.

Recall that, given a graph $G$, we denote by $d^G_t$ the number of vertices in $G$ of degree $t$. The main goal of this section is to prove the following recursive formula for $m_e(G\square H_{k,\ell},r)$.

\begin{restatable}{thm}{GenJCyclesExact}
    \thlabel{Gen joined cycles exact}
Let $k\geq \ell\geq 4$ and $r>1$ be integers. If there exists a proper edge-colouring $c$ of $G$ such that $m_e(G,i)=\dim(W^{i}_{G,c})$ for each $r-2\leq i\leq r$, then  
\begin{align*}
    m_e(G\square H_{k,\ell},r) &= m_e(G,r) + (k+\ell-5) m_e(G,r-1) + 2m_e(G,r-2) \\ &\hspace{4em} + d^G_{r-1} + (k+\ell-3)d^G_{r-2} + (k+\ell-1) \sum\limits_{i=0}^{r-3} d^G_i.
\end{align*}
\end{restatable}

As in the case of \thref{trees exact}, we provide a construction to give the upper bound in \thref{joined cycles upper}, and use properties of polynomials to provide the lower bound in \thref{joined cycles lower}, which, together with \thref{HHQ ineq}, directly imply Theorem~\ref{Gen joined cycles exact}. 

\begin{prop}
\thlabel{joined cycles upper}
Let $r>1$ and $k\geq \ell\geq 4$ be integers. For any graph $G$, 
\begin{align*}
m_e(G\square H_{k,\ell},r)&\leq m_e(G,r) + (k+\ell-5) m_e(G,r-1) + 2 m_e(G,r-2) + d^G_{r-1} \\ &\hspace{4em} + (k+\ell-3)d^G_{r-2} + (k+\ell-1)\sum\limits_{t=0}^{r-3} d^G_{t}.
\end{align*}
\end{prop}

\begin{prop}
\thlabel{joined cycles lower}
Let $r>1$ and $k\geq \ell\geq 4$ be integers, and let $c:E\rightarrow \R$ be a proper edge-colouring of a graph $G=(V,E)$. Then there exists a proper edge colouring $c^{\prime}$ of $G\square H_{k,\ell}$ for which 
\begin{align*}
    \dim\left(W^r_{G\square H_{k,\ell},c^\prime}\right) &\geq \dim\left(W^r_{G,c}\right) + (\ell+k-5) \dim\left(W^{r-1}_{G,c}\right) + 2 \dim\left(W^{r-2}_{G,c}\right) \\ &\hspace{4em}+ d^G_{r-1} + (\ell+k-3)d^G_{r-2} + (\ell+k-1)\sum\limits_{t=0}^{r-3} d^G_{t}.
\end{align*}
\end{prop}

Throughout this section, let $\{G_i:i\in I\}$ denote the $k+\ell-2$ copies of $G$ in $G\square H_{k,\ell}$ corresponding to the $k+\ell-2$ vertices in $H_{k,\ell}$. For each vertex $v\in V(G)$, denote its corresponding vertex in $G_i$ by $v_i$, where we sometimes refer to $v_i$ as $v_{i^\prime}$, for $i\in \{1,2\}$.

\begin{proof}[Proof of Proposition~\ref{joined cycles upper}]

First, we consider the case where $\delta(G)\geq r$. We construct an $r$-percolating set $F\subseteq E(G\square H_{k,\ell})$ for $G\square H_{k,\ell}$ as follows: Pick an optimal $r$-percolating set $F_1$ for $G_1$. For each $i\in I\setminus\{1,\ell,k^\prime\}$, pick an optimal $(r-1)$-percolating set $F_i$ on $G_i$. Finally, let $F_\ell$ and $F_{k^\prime}$ be optimal $(r-2)$-percolating sets on $G_\ell$ and $G_{k^\prime}$, respectively. Let $F=\bigcup_{i\in I} F_i$. See Figure \ref{fig:jcycle up} for an illustration of this construction.

We claim that $F$ percolates in $G\square H_{k,\ell}$. By our choice of $F_1$, after running the $r$-bond bootstrap percolation process on $G_1$, all edges in $G_1$ will be infected. Since $\delta(G)\geq r$, all edges between $G_1$ and each of $G_2$, $G_\ell$, and $G_{k^\prime}$ will now become infected. As each vertex $v_2\in V(G_2)$ has an infected edge coming from $G_1$, this, together with $F_2$, will infect the edges in $G_2$. Again, since $\delta(G)\geq r$, all edges between $G_2$ and both $G_3$ and $G_{3^\prime}$ will now become infected. As each vertex $v_3\in V(G_3)$ has an infected edge coming from $G_2$, this, together with $F_3$, will infect the edges in $G_3$. Continuing in this manner, we can infect $G_3,...,G_{\ell-1}$, and similarly $G_{3^\prime},...G_{(k-1)^\prime}$. Finally, since $\delta(G)\geq r$, all edges between $G_{\ell-1}$ and $G_\ell$ will become infected. Since each vertex $v_\ell\in V(G_\ell)$ has one infected edge coming from $G_{\ell-1}$ and one from $G_1$, this, together with $F_\ell$, will infect the edges in $G_\ell$. A similar argument shows that the edges in $G_{k^\prime}$ will also become infected. Hence $F$ percolates in $G\square H_{k,\ell}$.

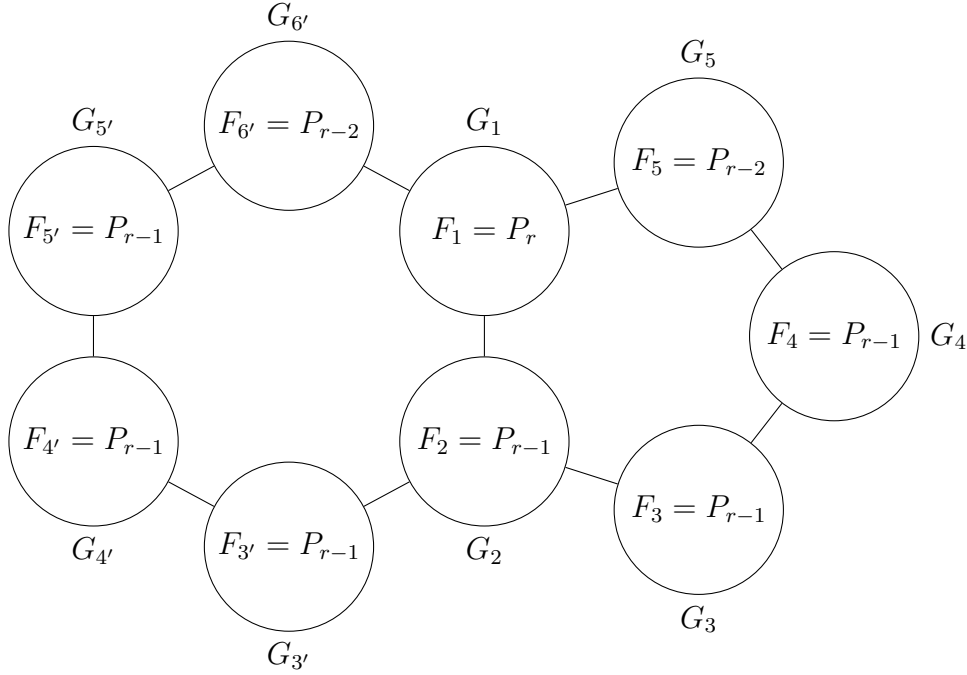
\begin{figure}[H]
\centering
\begin{tikzpicture}[xscale=1.5, yscale=1.4]
    \node [circ] (1) at (0,1)[label=above: {$G_1$}]{$F_1=P_r$};
    \node [circ] (2) at (0,-1)[label=below: {$G_2$}]{$F_2=P_{r-1}$};
    \node [circ] (3) at (1.9,-1.65)[label=below: $G_3$]{$F_3=P_{r-1}$};
    \node [circ] (4) at (3.1,0)[label=right: $G_4$]{$F_4=P_{r-1}$};
    \node [circ] (5) at (1.9,1.65)[label=above: $G_5$]{$F_5=P_{r-2}$};
    \node [circ] (3') at (-1.732,-2)[label=below: $G_{3'}$]{$F_{3'}=P_{r-1}$};
    \node [circ] (4') at (-3.464,-1)[label=below: $G_{4'}$]{$F_{4'}=P_{r-1}$};
    \node [circ] (5') at (-3.464,1)[label=above: $G_{5'}$]{$F_{5'}=P_{r-1}$};
    \node [circ] (6') at (-1.732,2)[label=above: $G_{6'}$]{$F_{6'}=P_{r-2}$};

    \draw (1)--(2);
    \draw (2)--(3);
    \draw (3)--(4);
    \draw (4)--(5);
    \draw (5)--(1);
    \draw (2)--(3');
    \draw (3')--(4');
    \draw (4')--(5');
    \draw (5')--(6');
    \draw (6')--(1);
\end{tikzpicture}
\caption[The $r$-percolating set $F$ for $G\square H_{6,5}$ when $\delta(G)\geq r$.]{The $r$-percolating set $F$ for $G\square H_{6,5}$ when $\delta(G)\geq r$. $P_i$ denotes an optimal $i$-percolating set on $G$.}
\label{fig:jcycle up}
\end{figure}

It remains to consider the case where $\delta(G)<r$. Let the set $F'$ be obtained from $F$ by adding the following edges. 
\begin{itemize}[nosep]
    \item For each $v\in V(G)$ with $\deg_G(v)=r-1$, add the edge $v_1v_2$ to the set $F$. Adding this edge will guarantee that, once $G_1$ is fully infected, the infection will spread to the edges incident with $v_2$ in $G_2$, after which the process continues as described above. Note that this adds an additional $d^G_{r-1}$ edges to the set $F$.
    \item For each $v\in V(G)$ with $\deg_G(v)=r-2$, add all edges $v_iv_j$, except $v_1v_{k^\prime}$ and $v_2v_3$, to the set $F$. Note that the edges $v_1v_{k^\prime}$ and $v_2v_3$ will become infected once the edges in $G_1$ and $G_2$ are infected, respectively, allowing the process to continue as described above. This adds an additional $(k+\ell-3)d^G_{r-2}$ edges to the set $F$. 
    \item For each $v\in V(G)$ with $\deg_G(v)\leq r-3$, add all edges $v_iv_j$ to the set $F$. This adds an additional $(k+\ell-1)\sum_{t=0}^{r-3} d^G_{t}$ edges to the set $F$.
\end{itemize} 

Altogether, we added $d^G_{r-1} + (k+\ell-3)d^G_{r-2} + (k+\ell-1)\sum_{t=0}^{r-3} d^G_{t}$ edges to the $m_e(G,r) + (k+\ell-5) m_e(G,r-1) + 2 m_e(G,r-2)$ edges initially in $F$. Since the resulting set $F$ percolates in $G$, we have established the desired upper bound on $m_e(G\square H_{k,\ell},r)$.  
\end{proof}

\begin{proof}[Proof of \thref{joined cycles lower}] 
 We refer to the copy of $G\square C_\ell$ in $G\square H_{k,\ell}$ as $C$, and the copy of $G\square C_k$ in $G\square H_{k,\ell}$ as $C^\prime$. Note that we take the indices in $C$ to be modulo $\ell$, and those in $C^\prime$ to be modulo $k$, where we perform addition on $[k^\prime]$ irrespective of the superscripts; for example, $3^\prime + 1 = (3+1)^\prime = 4^\prime$. 

Let $\{\alpha_i:i\in I\cup\{2^\prime\}\}$ be distinct real numbers not in $c(E)$. For convenience, we sometimes denote $\alpha_1$ by $\alpha_{1^\prime}$. Let $c^{\prime}$ be the proper edge-colouring of $G\square H_{k,\ell}$ defined as follows: Let $c^{\prime}$ be consistent with $c$ on each $G_i$, where $i\in I$. For each $v\in V(G)$ and $i\in [\ell]$, let $c^\prime(v_iv_{i+1})=\alpha_i$. For each $v\in V(G)$ and $i\in [k^\prime]\setminus\{1^\prime\}$, let $c^\prime(v_iv_{i+1})=\alpha_i$. See Figure \ref{JCycle} for an illustration of this edge-colouring of $G\square H_{k,\ell}$.

\begin{figure}[H]
\centering
\begin{tikzpicture}[scale=1.3]
    \node [circl] (1) at (0,1)[label=above: {$G_1$}]{$c$};
    \node [circl] (2) at (0,-1)[label=below: {$G_2$}]{$c$};
    \node [circl] (3) at (1.9,-1.65)[label=below: $G_3$]{$c$};
    \node [circl] (4) at (3.1,0)[label=right: $G_4$]{$c$};
    \node [circl] (5) at (1.9,1.65)[label=above: $G_5$]{$c$};
    \node [circl] (3') at (-1.732,-2)[label=below: $G_{3'}$]{$c$};
    \node [circl] (4') at (-3.464,-1)[label=below: $G_{4'}$]{$c$};
    \node [circl] (5') at (-3.464,1)[label=above: $G_{5'}$]{$c$};
    \node [circl] (6') at (-1.732,2)[label=above: $G_{6'}$]{$c$};

    \draw (1)--(2);
    \draw (2)--(3);
    \draw (3)--(4);
    \draw (4)--(5);
    \draw (5)--(1);
    \draw (2)--(3');
    \draw (3')--(4');
    \draw (4')--(5');
    \draw (5')--(6');
    \draw (6')--(1);

    \node [] (a1) at (0.2,0)[]{$\alpha_1$};
    \node [] (a2) at (1.05,-1.15)[]{$\alpha_2$};
    \node [] (a3) at (2.25,-0.75)[]{$\alpha_3$};
    \node [] (a4) at (2.25,0.75)[]{$\alpha_4$};
    \node [] (a5) at (1.05,1.15)[]{$\alpha_5$};
    \node [] (a2') at (-1,-1.35)[]{$\alpha_{2'}$};
    \node [] (a3') at (-2.5,-1.35)[]{$\alpha_{3'}$};
    \node [] (a4') at (-3.2,0)[]{$\alpha_{4'}$};
    \node [] (a5') at (-2.5,1.35)[]{$\alpha_{5'}$};
    \node [] (a6') at (-1,1.35)[]{$\alpha_{6'}$};
\end{tikzpicture}
\caption{The graph $G\square H_{6,5}$ and its edge colouring $c'$.}
\label{JCycle}
\end{figure}
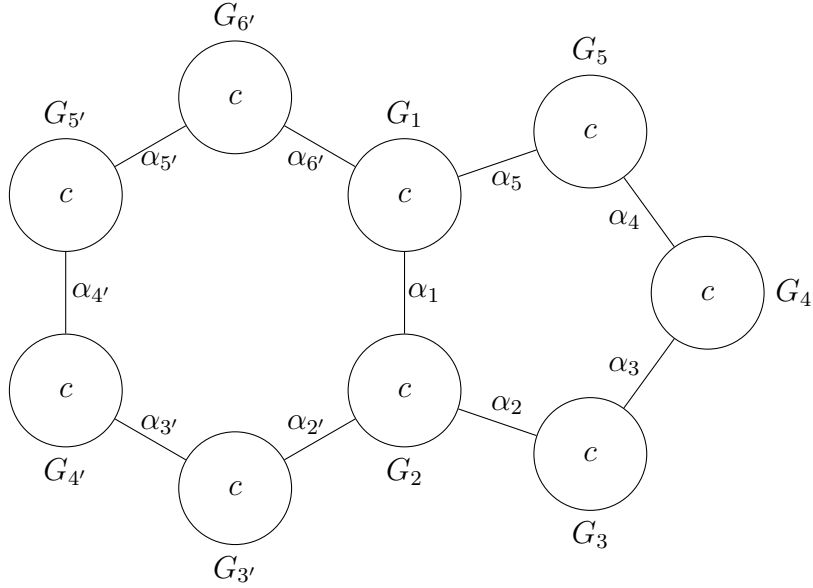

Let 
\begin{align*}
    N&:=\dim\left(W^r_{G,c}\right) + (\ell+k-5) \dim\left(W^{r-1}_{G,c}\right) + 2 \dim\left(W^{r-2}_{G,c}\right) + d^G_{r-1} \\ &\hspace{4em} + (\ell+k-3)d^G_{r-2}
    + (\ell+k-1)\sum\limits_{t=0}^{r-3} d^G_{t}.
\end{align*}
To prove this result, we will find a set of $N$ linearly independent vectors in $W^r_{G\square H_{k,\ell},c^\prime}$. We begin by finding a set $A$ of $\dim\left(W^r_{G,c}\right) + (\ell+k-5) \dim\left(W^{r-1}_{G,c}\right) + 2 \dim\left(W^{r-2}_{G,c}\right)$ linearly independent vectors in $W^r_{G\square H_{k,\ell},c^\prime}$.

First, consider a basis $B^{(r)}$ for $W^r_{G,c}$. For each vector $\q=(q_v)_{v\in V(G)}$ in $B^{(r)}$, define the vector $\p_\q^{1}=(p_u)_{u\in V(G\square H_{k,\ell})}$ so that $p_{v_i}=q_v$ for all $i\in I$ and $v\in V(G)$. Trivially, the two conditions in \thref{W} are satisfied; that is, $\deg(p_v)\leq \min\{r,\deg(v)\}-1$ for all vertices $v\in V(G\square H_{k,\ell})$, and $p_u(c_{uv})=p_v(c_{uv})$ for all edges $uv\in E(G\square H_{k,\ell})$. Hence $\p_\q^{1}\in W^r_{G\square H_{k,\ell},c^{\prime}}$. Let $A_1=\{\p_\q^{1}:q\in B^{(r)}\}$. Note that the restriction of any vector $\p_\q^{1}\in A$ to $G_1$ equals $\q$, and thus, since $B^{(r)}$ is a basis, the vectors in $A_1$ are linearly independent.

Now, consider a basis $B^{(r-1)}$ for $W^{r-1}_{G,c}$. First, fix $i\in I\setminus \{1,2,\ell,k^\prime\}$. For each vector $\q=(q_v)_{v\in V(G)}\in B^{(r-1)}$, define $\p_\q^{i}=(p_u)_{u\in V(G\square H_{k,\ell})}$ so that, for all $v\in V(G)$,
\[p_{v_i}=\frac{x-\alpha_{i-1}}{\alpha_i-\alpha_{i-1}} q_v, \hspace{0.5cm}p_{v_{i+1}}=\frac{x-\alpha_{i+1}}{\alpha_i-\alpha_{i+1}} q_v, \text{ and }\hspace{0.5cm} p_{v_j}\equiv 0 \text{ for all other } j\in I.\]
It is not difficult to check that $\p_\q^{i}\in W^{r}_{G\square H_{k,\ell},c^\prime}$. Now, for each $i\in I\setminus \{1,2,\ell,k^\prime\}$, define $A_i:=\{\p_\q^{i}:\q\in B^{(r-1)}\}$. Note that the restriction of any vector $\p_\q^{i}\in A_i$ to $G_i$ is $\frac{x-\alpha_{i-1}}{\alpha_i-\alpha_{i-1}} \q$. Therefore, since $x-\alpha_{i-1}\not\equiv 0$ and $B^{(r-1)}$ is a basis, the vectors in $A_i$ are linearly independent.

Now, for each $\q=(q_v)_{v\in V(G)}$ in $B^{(r-1)}$, define the vector $\p_\q^{2}=(p_u)_{u\in V(G\square H_{k,\ell})}$ so that, for all $v\in V(G)$, 
\[p_{v_3}=\frac{x-\alpha_3}{(\alpha_2-\alpha_3)(\alpha_{2^\prime}-\alpha_1)} q_v, \hspace{0.4cm}p_{v_2}=\frac{x-\alpha_1}{(\alpha_2-\alpha_1)(\alpha_{2^\prime}-\alpha_1)} q_v,\hspace{0.4cm} p_{v_{3^\prime}}=\frac{x-\alpha_{3^\prime}}{(\alpha_2-\alpha_1)(\alpha_{2^\prime}-\alpha_{3^\prime})} q_v,  \]
and $p_{v_i}\equiv 0$ for all other $i\in I$. It is not difficult to check that $\p_\q^{2}\in W^{r}_{G\square H_{k,\ell},c^\prime}$. We define $A_2:=\{\p_\q^{2}:\q\in B^{(r-1)}\}$. Note that the restriction of any vector $\p_\q^{2}\in A_2$ to $G_2$ is $\frac{x-\alpha_1}{(\alpha_2-\alpha_1)(\alpha_{2^\prime}-\alpha_1)} \q$. Therefore, since $x-\alpha_1\not\equiv 0$ and $B^{(r-1)}$ is a basis, the vectors in $A_2$ are linearly independent.

Finally, consider a basis $B^{(r-2)}$ for $W^{r-2}_{G,c}$. Fix $i\in \{\ell,k^\prime\}$. For each $\q=(q_v)_{v\in V(G)}$ in $B^{(r-2)}$, define the vector $\p_\q^{i}=(p_u)_{u\in V(G\square H_{k,\ell})}$ such that $p_{v_i}=(x-\alpha_{i-1})(x-\alpha_i)q_v$ for all $v$ in $V(G)$, and $p_{u}\equiv 0$ for all other $u\in V(G\square H_{k,\ell})$. It is straightforward to see that $\p_\q^{i}\in W^{r}_{G\square H_{k,\ell},c^\prime}$. For $i\in \{\ell,k^\prime\}$, let $A_i:=\{\p_\q^{i}:\q\in B^{(r-2)}\}$. Note that the restriction of any vector $\p_\q^{i}\in A_i$ to $G_i$ is $(x-\alpha_{i-1})(x-\alpha_i)\q$. Thus, since $(x-\alpha_{i-1})(x-\alpha_i)\not\equiv 0$ and $B^{(r-2)}$ is a basis, the vectors in $A_i$ are linearly independent.

Let $A:=\bigcup_{i\in I} A_i$. Then $|A|=\dim\left(W^r_{G,c}\right) + (\ell+k-5) \dim\left(W^{r-1}_{G,c}\right) + 2 \dim\left(W^{r-2}_{G,c}\right)$ as $|A_1|=\dim(W_{G,c}^r)$, $|A_i|=\dim(W_{G,c}^{r-1})$  for $i\in I\setminus \{1,\ell,k'\}$, and $|A_i|=\dim(W_{G,c}^{r-2})$ for $i\in \{\ell,k'\}$.

\begin{claim}
\thlabel{claim:jcycles A ind}
    The vectors in $A$ are linearly independent. 
\end{claim}

\noindent\emph{Proof of \thref{claim:jcycles A ind}.} Suppose for a contradiction that there exists a linear combination $S=\sum_{\p\in A} \beta_\p\p = 0$, where $\beta_\p\neq0$ for some $\p\in A$. Write $S=\sum_{i\in I}S_i$ where $S_i=\sum_{\p\in A_i} \beta_\p\p$ for each $i\in I$. Let $j\in I$ be the smallest index, relative to the ordering $\phi$, such that there exists $\p\in A_j$ with $\beta_\p\neq 0$. 

Suppose first that $j=1$. Then, for each $v\in V(G)$, we have $p_{v_1}\equiv 0$ for each vector $\p\in \bigcup_{i\in I\setminus\{1\}}A_i$ by definition, but $p_{v_1}\not\equiv 0$ for all non-zero $\p\in A_1$. Hence the restriction of the sum $S$ to the coordinates corresponding to vertices in $G_1$ is equal to the restriction of $S_1$ to $G_1$. However, since the vectors in $A_1$ are linearly independent, this contradicts our assumption that $\beta_\p\neq 0$ for some $\p\in A_1$. Hence $j\neq 1$, and thus $\beta_\p=0$ for all $\p\in A_1$. 

Next, assume $j=2$. Similarly to the previous case, for each $v\in V(G)$, we have $p_{v_2}\equiv 0$ for all $\p\in \bigcup_{i\in I\setminus\{1,2\}}A_i$, but $p_{v_2}\not\equiv 0$ for all non-zero $\p\in A_2$. Since $\beta_\p=0$ for all $\p\in A_1$, the restriction of $S$ to the coordinates corresponding to vertices in $G_2$ is equal to the restriction of $S_2$ to $G_2$. However, since the vectors in $A_2$ are linearly independent, this contradicts our assumption that $\beta_\p\neq 0$ for some $\p\in A_2$. Hence $j\neq 2$, and thus $\beta_\p=0$ for all $\p\in A_2$. 

Continuing in this manner, we find that $\beta_\p=0$ for all $\p\in A$, a contradiction. Therefore, the vectors in $A$ are linearly independent.~$_{(\square)}$

To achieve the desired lower bound on $\dim\left(W^r_{G\square H_{k\ell},c^\prime}\right)$, we must extend $A$ to include an additional $d^G_{r-1} + (\ell+k-3)d^G_{r-2} + (\ell+k-1)\sum_{t=0}^{r-3} d^G_{t}$ vectors, all of which are linearly independent. To this end, consider the set of vectors $Z_c\in \R^{|V(G)|}$ as defined in \thref{Z}.

\begin{claim}
\thlabel{jcycles claim}
For each non-zero $\p\in \vspan(A)$, there exists $\z\in Z_c$ such that $\p(\z)\neq 0$. 
\end{claim}

\noindent\emph{Proof of Claim:} Let $\p= (p_u)_{u\in V(G\square H_{k,\ell})}$ be a non-zero vector in $\vspan(A)$. Let $j$ be the smallest index, relative to the ordering $\phi$, for which there exists $v_j\in V(G_j)$ with $p_{v_j}\not\equiv 0$. By definition of $A$, if $j=1$, we can write $p_{v_1}=q_v$ for some $\q\in (q_u)_{u\in V(G)}$ in $\vspan(B^{(r)})=W^r_{G,c}$. Similarly, if $j\in I\setminus\{1,\ell,k'\}$, then we can write $p_{v_j}=q_v(x-\alpha_{j-1})$ for some $\q\in (q_u)_{u\in V(G)}$ in $\vspan(B^{(r-1)})=W^{r-1}_{G,c}$. Finally, if $j\in \{\ell,k'\}$, then we can write $p_{v_j}=q_v(x-\alpha_{j-1})(x-\alpha_j)$ for some $\q\in (q_u)_{u\in V(G)}$ in $\vspan(B^{(r-2)})=W^{r-2}_{G,c}$. In any case, by \thref{W}, we have $\deg(q_v)\leq \deg_G(v)-1$. Therefore, since the colours $\alpha_i$ were chosen to be distinct from those in $c(E)$, by \thref{zeros lemma}, there exists $\z\in Z_p$ such that $\p(\z)\neq 0$.~$_{(\square)}$

Therefore, in order to finish the proof of \thref{joined cycles lower}, it suffices to find a set $X$ of $d^G_{r-1} + (\ell+k-3)d^G_{r-2} + (\ell+k-1)\sum_{t=0}^{r-3} d^G_{t}$ linearly independent vectors in $W^r_{G\square H_{k,\ell}, c^\prime}$, each of which evaluates to zero on all of $Z_c$. This, together with \thref{jcycles claim}, will guarantee that these vectors are independent from those in $A$. We will define families of vectors in $W^r_{G\square H_{k,\ell},c^\prime}$ with the property that they each evaluate to $0$ on all of $Z_c$.

First, for each $v\in V(G)$ with $\deg_G(v)=r-1$, define $\p_v = (p_u)_{u\in V(G\square H_{k,\ell})}$ such that 
\vspace{-2mm}
\[p_{v_i}=\prod_{u\in N_G(v)} (x-c_{uv}) \text{ for all } i\in I,\]
and $p_u\equiv 0$ for all other $u\in V(G\square H_{k,\ell})$. 
Define $X_v:=\{\p_v\}$. Let $X_1:=\bigcup X_v$, where the union is over all $v\in V(G)$ such that $\deg_G(v)=r-1$. Note that $|X_1|=d^G_{r-1}$.

Now, let $v\in V(G)$ with $\deg_G(v)=r-2$. First, for each $i\in I\setminus \{1,2,\ell,k^\prime\}$, define the vector $\p_v^{i}= (p_u)_{u\in V(G\square H_{k,\ell})}$ as follows: 
\begin{itemize}[nosep]
    \item Let $p_{v_i}$ be the unique polynomial of degree $\deg(v_i)-1$ that evaluates to $1$ at $\alpha_i$ and evaluates to zero at all $\deg(v_i)-1$ other colours of the edges incident with $v_i$ in $G\square H_{k,\ell}$. That is, $p_{v_i}=\beta_i\prod_{u\in N(v_i)\setminus\{v_{i+1}\}} (x-c'_{v_iu})$, where $\beta_i$ is a non-zero constant chosen so that $p_{v_i}(\alpha_i)=1$.
    \item Let $p_{v_{i+1}}$ be the unique polynomial of degree $\deg(v_{i+1})-1$ that evaluates to $1$ at $\alpha_i$ and evaluates to zero at all $\deg(v_{i+1})-1$ other colours of the edges incident with $v_{i+1}$ in $G\square H_{k,\ell}$. That is, $p_{v_{i+1}}=\beta_{i+1}\prod_{u\in N(v_{i+1})\setminus\{v_{i}\}} (x-c'_{v_{i+1}u})$, where $\beta_{i+1}$ is a non-zero constant chosen so that $p_{v_{i+1}}(\alpha_i)=1$.
    \item Let $p_u\equiv 0$ for all $u\notin\{v_i,v_{i+1}\}$. 
\end{itemize}

Next, define $\p_v^{2}= (p_u)_{u\in V(G\square H_{k,\ell})}$ so that $p_u\equiv 0$ for all $u\notin\{v_2,v_3,v_{3^\prime}\}$, 
\vspace{-2mm}
\[p_{v_2}=(x-\alpha_1)\!\!\!\!\!\!\prod_{u\in N_G(v)}\!\!\!\!\! (x-c_{uv}),\  p_{v_3}=\beta(x-\alpha_3)\!\!\!\!\!\!\prod_{u\in N_G(v)} \!\!\!\!\!(x-c_{uv}),\text{ and } p_{v_{3^\prime}}=\beta^\prime(x-\alpha_{3^\prime})\!\!\!\!\!\!\prod_{u\in N_G(v)} \!\!\!\!\!(x-c_{uv}),\]
\noindent where $\beta$ is a non-zero constant chosen so that $p_{v_3}(\alpha_2) = p_{v_2}(\alpha_2)$, 
and $\beta^\prime$ is a non-zero constant chosen so that $p_{v_{3^\prime}}(\alpha_{2^\prime}) = p_{v_2}(\alpha_{2^\prime})$. Similarly, define $\p_v^{1}= (p_u)_{u\in V(G\square H_{k,\ell})}$ so that $p_u\equiv 0$ for all $u\notin\{v_1,v_\ell,v_{k^\prime}\}$,
\vspace{-2mm}
\[p_{v_1}=(x-\alpha_1)\!\!\!\!\!\!\prod_{u\in N_G(v)} \!\!\!\!(x-c_{uv}), \ p_{v_\ell}=\beta(x-\alpha_{\ell-1})\!\!\!\!\!\!\prod_{u\in N_G(v)} \!\!\!\!(x-c_{uv}), \text{ and } p_{v_{k^\prime}}=\beta^\prime(x-\alpha_{(k-1)^\prime})\!\!\!\!\!\!\prod_{u\in N_G(v)} \!\!\!\!(x-c_{uv}),\]
where $\beta$ is a non-zero constant chosen so $p_{v_\ell}(\alpha_\ell) = p_{v_1}(\alpha_\ell)$, 
and $\beta^\prime$ is a non-zero constant chosen so $p_{v_{k^\prime}}(\alpha_{k^\prime}) = p_{v_1}(\alpha_{k^\prime})$. 

Finally, define $\p_v^{0}= (p_u)_{u\in V(G\square H_{k,\ell})}$ so that $p_u\equiv 0$ for all $u\notin\{v_1,v_2,v_{3^\prime},v_{k^\prime}\}$,
\vspace{-2mm}
\[p_{v_1}=\frac{x-\alpha_{\ell}}{\alpha_1-\alpha_{\ell}}\prod\limits_{u\in N_G(v)} (x-c_{uv}), \ \ p_{v_2}=\frac{x-\alpha_{2}}{\alpha_1-\alpha_{2}}\prod\limits_{u\in N_G(v)} (x-c_{uv}),\]
\[p_{v_{3^\prime}}=\beta(x-\alpha_{3^\prime})\prod\limits_{u\in N_G(v)} (x-c_{uv}), \ \text{ and } \ p_{v_{k^\prime}}=\beta^\prime(x-\alpha_{(k-1)^\prime})\prod\limits_{u\in N_G(v)} (x-c_{uv}),\]
where $\beta$ is a non-zero constant chosen so $p_{v_{3^\prime}}(\alpha_{2^\prime}) = p_{v_2}(\alpha_{2^\prime})$, and 
$\beta^\prime$ is a non-zero constant chosen so $p_{v_{k^\prime}}(\alpha_{k^\prime}) = p_{v_1}(\alpha_{k^\prime})$.

Now, for each $v\in V(G)$ with $\deg_G(v)=r-2$, we define $X_v:=\{\p_v^i: i\in \{0\}\cup I\setminus \{\ell,k^\prime\}\}$. Let $X_2:=\bigcup X_v$, where the union is taken over all $v\in V(G)$ with $\deg_G(v)=r-2$. Note that, since $|I\setminus \{\ell,k^\prime\}|=k+\ell-3$, we have $|X_2|=(\ell+k-3)d_{r-2}^G$.

Now suppose $v\in V(G)$ with $\deg_G(v)\leq r-3$. First, for each $i\in [\ell]$, define the vector $\p_v^{i}= (p_u)_{u\in V(G\square H_{k,\ell})}$ as follows: 
\begin{itemize}[nosep]
    \item Let $p_{v_i}$ be the unique polynomial of degree $\deg(v_i)-1$ that evaluates to $1$ at $\alpha_i$ and evaluates to zero on all $\deg(v_i)-1$ other colours of the edges incident with $v_i$ in $G\square H_{k,\ell}$. That is, $p_{v_i}=\beta_i\prod_{u\in N(v_i)\setminus\{v_{i+1}\}} (x-c'_{v_iu})$, where $\beta_i$ is a non-zero constant chosen so that $p_{v_i}(\alpha_i)=1$. 
    \item Let $p_{v_{i+1}}$ be the unique polynomial of degree $\deg(v_{i+1})-1$ that evaluates to $1$ at $\alpha_i$ and evaluates to zero on all $\deg(v_{i+1})-1$ other colours of the edges incident with $v_{i+1}$ in $G\square H_{k,\ell}$. That is, $p_{v_{i+1}}=\beta_{i+1}\prod_{u\in N(v_{i+1})\setminus\{v_{i}\}} (x-c'_{v_{i+1}u})$, where $\beta_{i+1}$ is a non-zero constant chosen so that $p_{v_{i+1}}(\alpha_{i})=1$.
    \item Let $p_u\equiv 0$ for all $u\notin\{v_i,v_{i+1}\}$.
\end{itemize}

Similarly, for $i\in [k^\prime]\setminus \{1^\prime\}$, define the vector $\p_v^{i}= (p_u)_{u\in V(G\square H_{k,\ell})}$ as follows: 
\begin{itemize}[nosep]
    \item Let $p_{v_i}$ be the unique polynomial of degree $\deg(v_i)-1$ that evaluates to $1$ at $\alpha_i$ and evaluates to zero on all $\deg(v_i)-1$ other colours of the edges incident with $v_i$ in $G\square H_{k,\ell}$. 
    \item Let $p_{v_{i+1}}$ be the unique polynomial of degree $\deg(v_{i+1})-1$ that evaluates to $1$ at $\alpha_i$ and to zero on all $\deg(v_{i+1})-1$ other colours of the edges incident with $v_{i+1}$ in $G\square H_{k,\ell}$.
    \item Let $p_u\equiv 0$ for all $u\notin\{v_i,v_{i+1}\}$. 
\end{itemize} 

For each $v\in V(G)$ with $\deg_G(v)\leq r-3$, define $X_v:=\{\p_v^i: i\in I\cup\{2^\prime\}\}$. Let $X_3:=\bigcup X_v$, where the union is taken over all $v\in V(G)$ such that $\deg_G(v)\leq r-3$. Note that, since $|I\cup\{2^\prime\}|=\ell+k-1$, we have $|X_3|= (\ell+k-1)\sum_{t=0}^{r-3}d_t^G$.

It is straightforward to check that the vectors defined above are all in $W^r_{G\square H_{k,\ell},c^\prime}$ and have the property that they each evaluate to $0$ on all of $Z_c$. Define $X:=X_1\cup X_2\cup X_3$. Note that 
\[|X| = |X_1|+|X_2|+|X_3|=d^G_{r-1} + (\ell+k-3)d^G_{r-2} + (\ell+k-1)\sum_{t=0}^{r-3} d^G_{t}.\] 

\begin{claim}
\thlabel{claim:jcycles X ind}
    The vectors in $X$ are linearly independent. 
\end{claim}

\noindent\emph{Proof of \thref{claim:jcycles X ind}.} Since each vector in $X$ is only non-zero in the entries corresponding to a unique vertex $v\in V(G)$, it suffices to show that the vectors in each set $X_v$ are linearly independent, where $\deg_G(v)\leq r-1$. When $\deg_G(v)=r-1$, we have $|X_v|=1$, and thus $X_v$ is clearly linearly independent. So assume $\deg_G(v)<r-1$. 

First, suppose $\deg_G(v)=r-2$. Suppose for a contradiction that $S=\sum_{\p\in X_v} \gamma_\p \p = 0$ is a non-trivial linear combination. For ease of notation, we write $\gamma_i$ for $\gamma_{\p_v^i}$. First, suppose $\gamma_i\neq 0$ for some $i\in I\setminus\{1,2,3,3^\prime,\ell,k^\prime\}$. Note that $\p_v^i$ and $\p_v^{i-1}$ are the only two vectors in $X_v$ for which $p_{v_i}\not\equiv 0$. Hence $\gamma_{i-1}\neq 0$. Moreover, we have $\gamma_{i-1}\left. \p_v^{i-1}\right\vert_{v_i} + \gamma_{i}\left. \p_{v}^i\right\vert_{v_i} \equiv 0$. But $\left. \p_v^{i-1}\right\vert_{v_i}$ evaluated at $\alpha_i$ is $0$, while $\left. \p_v^{i}\right\vert_{v_i}$ evaluated at $\alpha_i$ is $1$, a contradiction. Hence $\gamma_i=0$ for all $i\in I\setminus\{1,2,3,3^\prime,\ell,k^\prime\}$, and thus $\gamma_i\neq 0$ for some $i\in \{0,1,2,3,3'\}$. Now, since $p_{v_{4}}\not\equiv 0$ in $\p_v^{3}$, but $p_{v_{4}}\equiv 0$ in $\p_v^{i}$ for all $i\in \{0,1,2,3^\prime\}$, we have $\gamma_{3}=0$. Similarly, $\gamma_{3^\prime}=0$. Thus $S=\gamma_0\p_v^0 + \gamma_1\p_v^1 + \gamma_2\p_v^2$. Now, since $p_{v_3}\not\equiv 0$ in $\p_v^{2}$, but $p_{v_{3}}\equiv 0$ in $\p_v^{i}$ for all $I\in \{0,1\}$, we have $\gamma_{2}=0$. Similarly, $\gamma_{1}=0$. But then $S=\gamma_0\p_v^0=0$ with $\p_v^0\neq 0$, and thus $\gamma_0=0$, contradicting $S$ being a non-trivial linear combination. Hence $X_v$ is linearly independent.

Finally, let $\deg_G(v)\leq r-3$. Suppose for a contradiction that $S=\sum_{\p\in X_v} \gamma_\p \p = 0$ is a non-trivial linear combination. As before, we write $\gamma_i$ for $\gamma_{\p_v^i}$. First, suppose $\gamma_i\neq 0$ for some $i\in I\setminus\{1,2\}$. Note that $\p_v^i$ and $\p_v^{i-1}$ are the only two vectors in $X_v$ for which $p_{v_i}\not\equiv 0$. Hence $\gamma_{i-1}\neq 0$. Moreover, we have $\gamma_{i-1}\left. \p_v^{i-1}\right\vert_{v_i} + \gamma_{i}\left. \p_{v}^i\right\vert_{v_i} \equiv 0$. But $\left. \p_v^{i-1}\right\vert_{v_i}$ evaluated at $\alpha_i$ is $0$, while $\left. \p_v^{i}\right\vert_{v_i}$ evaluated at $\alpha_i$ is $1$, a contradiction. Hence $\gamma_i=0$ for all $i\in I\setminus\{1,2\}$, and thus $S= \gamma_1\p_v^1 + \gamma_2\p_v^2 + \gamma_{2^\prime}\p_v^{2^\prime}$. Now, since $p_{v_3}\not\equiv 0$ in $\p_v^{2}$, but $p_{v_3}\equiv 0$ in $\p_v^{i}$ for all $i\in \{1,2^\prime\}$, we have $\gamma_{2}=0$. Similarly, $\gamma_{2^\prime}=0$. But then $S=\gamma_1\p_v^1=0$ with $\p_v^1\neq 0$, and thus $\gamma_1=0$, contradicting $S$ being a nontrivial linear combination. Hence $X_v$ is linearly independent.~$_{(\square)}$

Let $Y=A\cup X$. Therefore, by \thref{claim:jcycles A ind,jcycles claim,claim:jcycles X ind}, we have found
\vspace{-2mm}
\[\dim\left(W^r_{G,c}\right) + (\ell+k-5) \dim\left(W^{r-1}_{G,c}\right) + 2 \dim\left(W^{r-2}_{G,c}\right) + d^G_{r-1} + (\ell+k-3)d^G_{r-2} + (\ell+k-1)\sum\limits_{t=0}^{r-3} d^G_{t}\]
linearly independent vectors in $W_{G\square H_{k,\ell},c'}^r$, as required.
\end{proof}

\section{Conclusions}
\label{Sec:cons}

We now provide consequences of the results in the previous sections for specific products of graphs. First,
\thref{trees exact} immediately yields the following corollary.

\begin{cor}
\thlabel{trees exact path}
    Let $P_n$ be a path on $n\geq 2$ vertices, and let $r\in\Z^+$. If there exists a proper edge-colouring $c$ of $G$ such that $m_e(G,i)= \dim(W_{G,c}^i)$ for all $r-1\leq i\leq r$, then 
    \vspace{-2mm}
    \begin{align*}
        m_e(G\square P_n,r)&= m_e(G,r) + (n-1)m_e(G,r-1) + d^G_{r-1} + (n-1)\sum\limits_{t=0}^{r-2}d_t^G.
    \end{align*} 
\end{cor}

\vspace{-1mm}

\noindent Note that \thref{trees exact path} is a generalization of Theorem 9 from \cite{HHQ}, since the condition $m_e(G,i)= \dim(W_{G,c}^i)$ for all $i\geq 0$ holds when $G$ is a product of paths (see the proof of Theorem 9 in \cite{HHQ}). Moreover, as the edge-colouring $c'$ of $G\square P_n$ defined in the proof of \thref{trees lower} does not depend on the value of $r$, the proof of \thref{trees exact path} also gives the following result.

\begin{cor}
\thlabel{G path equal}
Let $P_n$ be a path on $n\geq 2$ vertices. If there exists a proper edge-colouring $c$ of $G$ such that $m_e(G,r)= \dim(W_{G,c}^r)$ for all $r\geq 0$, then there exists a proper edge-colouring $c'$ of $G\square P_n$ such that $m_e(G\square P_n,r)= \dim(W_{G\square P_n,c'}^r)$ for all $r\geq 0$.
\end{cor}

In the case of stars, since the edge-colouring $c'$ of $G\square S_k$ defined in proof of \thref{stars lower} (see Appendix A for details) does not depend on the value of $r$, the proof of \thref{Gen Stars exact} implies the following result.  

\begin{cor}
\thlabel{G star equal}
Let $k\in\Z^+$ and $G$ be a graph. If there exists a proper edge-colouring $c$ of $G$ such that $m_e(G,r)=\dim(W^{r}_{G,c})$ for each $r\geq 0$, then there exists a proper edge-colouring $c'$ of $G\square S_k$ such that $m_e(G\square S_k,r)= \dim(W_{G\square S_k,c'}^r)$ for all $r\geq 0$.
\end{cor}

Similarly, the proof of \thref{Gen joined cycles exact} provides an analogous corollary for theta graphs.

\begin{cor}
\thlabel{G theta equal}
Let $k\geq \ell \geq 4$ be integers, and let $G$ be a graph. If there exists a proper edge-colouring $c$ of $G$ such that $m_e(G,r)=\dim(W^{r}_{G,c})$ for each $r\geq 0$, then there exists a proper edge-colouring $c'$ of $G\square H_{k,\ell}$ such that $m_e(G\square H_{k,\ell},r)= \dim(W_{G\square H_{k,\ell},c'}^r)$ for all $r\geq 0$.
\end{cor}

In particular, since $m_e(K_1,r)=0=\dim(W^{r}_{K_1,c})$ for any $r\geq 0$, where $c$ is the empty edge-colouring, recursively applying \thref{G path equal,G star equal,G theta equal} gives the following result for graphs $G$ which are a product of paths, stars, and theta graphs.

\begin{cor}
\thlabel{prod equal}
Let $p\ge 1$, and let $G_1,...,G_p$ be graphs such that each $G_i$ is either a path, a star, or a theta graph of the form $H_{k,\ell}$ for $k\ge \ell\ge 4$. Let $G=G_1\square\cdots\square G_p$. Then there exists a proper edge-colouring $c$ of $G$ such that $m_e(G,r)=\dim(W^r_{G,c})$ for all $r\ge 0$.  
\end{cor}

Combining \thref{prod equal} with \thref{trees exact path,Gen Stars exact,Gen joined cycles exact} allows us to obtain recursive formulas for the value of $m_e(G,r)$ when $G$ can be written a a product $G_1\square \cdots \square G_p$ where each $G_i$ is either a path, a star, or a theta graph of the form $H_{k\ell}$ where $k\ge \ell\ge 4$. In particular, if the graphs $G_i$ are all stars, then recursively applying \thref{Gen Stars exact,G star equal} we obtain \thref{Stars exact}, restated here for convenience.

\StarsExact*

Similarly, if the graphs $G_i$ are all theta graphs, then recursively applying \thref{Gen joined cycles exact,G theta equal} we obtain \thref{joined cycles exact}, restated here for convenience. 

\JCyclesExact*

\section{Concluding Remarks} 
\label{sec:concl}

In this section, we present some open problems and avenues for future research related to the study of weak saturation via properties of polynomials.

The proofs of lower bounds on $m_e(G,r)$ given in this paper, and those by Hambardzumyan, Hatami, and Qian in~\cite{HHQ}, all relied on the polynomial method of Hambardzumyan, Hatami, and Qian~\cite{HHQ} (see \thref{HHQ ineq}). Therefore, a natural line of inquiry would be to utilize \thref{HHQ ineq} to determine the value of $m_e(G,r)$ for other graphs $G$. In addition to continuing this investigation for Cartesian products of graphs, one could also consider other graph products, such as a tensor product or a strong product. 

Recall that, in Section \ref{Sec:Trees}, we provided upper and lower bounds on $m_e(G\square T,r)$ when $T$ is a tree, and showed that these bounds match in certain cases, namely when either $\delta(G)\geq r-2$ or $T$ is a path. It would be interesting to provide matching bounds for other families of trees. It is our belief that the upper bound presented in \thref{trees upper} is closer to the exact value of $m_e(G\square T,r)$. It is possible that a different choice of polynomials in the proof of \thref{trees lower} could have improved the lower bound.

\begin{qn}
    Is it possible to improve the lower bound on $m_e(G\square T,r)$ presented in \thref{trees lower} so it matches the upper bound presented in \thref{trees upper}?
\end{qn}

In order to apply our results to give a recursive formula for $m_e(G,r)$ when $G$ is a Cartesian product of certain graphs, we require the condition that there exists a proper edge colouring $c$ such that $m_e(G,r)=\dim(W^r_{G,c})$ for these graphs $G$. Therefore, it would be interesting to classify the graphs for which this property holds. 

\begin{qn}
    Under which conditions is $m_e(G,r)=\dim(W^r_{G,c})$ for a graph $G$ and a proper edge-colouring $c$ of $G$?
\end{qn}

As mentioned in the introduction, the $r$-bond boostrap percolation process is an example of a more general process called \emph{weak saturation}. Given graphs $G$ and $H$, a spanning subgraph $F$ of $G$ is \emph{weakly-}$(G,H)$\emph{-saturated} if $F$ is $H$-free, but there exists an ordering $e_1,...,e_m$ of $E(G)\setminus E(F)$, such that, for each $i\in[m]$, the addition of the edge $e_i$ to $F\cup\{e_1,...,e_{i-1}\}$ creates a new copy of $H$. The \emph{weak saturation number} of $H$ in $G$, denoted by $\wsat(G,H)$, is the minimum number of edges in such a graph~$F$; that is, \[\wsat(G,H) = \min\{|E(F)|: F\text{ is weakly-}(G,H)\text{-saturated}\}.\] Note that $m_e(G,r)=\wsat(G,K_{1,r+1})$.

Many results providing good lower bounds for weak saturation numbers have been obtained via algebraic methods; see, for example, \cite{A,BMM,F,K84,K85,L,MS,Pik}. In particular a general linear algebraic lemma of Balogh, Bollob\'{a}s, Morris, and Riordan~\cite{BBMR} has been used to prove several different results on weak saturation; see \cite{BBMR,KMM,MN18,MNS}.

\begin{lemma}
\thlabel{LLAL}
    \emph{\cite{BBMR}} Let $F$ be a graph, let $H$ be a subgraph of $F$, and let $W$ be a vector space. Suppose there exists a set $S=\{f_e:e\in E(F)\}\subseteq W$ of vectors in $W$ such that, for every copy $H'$ of $H$ in $F$, there exists a set of non-zero scalars $\{c_e:e\in E(H')\}$ such that $\sum_{e\in E(H')} c_ef_e=0$. Then $\wsat(F,H)\geq \dim(S)$.
\end{lemma}

Recently, Miralaei, Mohammadian and Tayfeh-Rezaie~\cite{MMT} showed that the polynomial method from~\cite{HHQ} is a special case of Lemma~\ref{LLAL}. That is, any bound obtained by applying Theorem~\ref{HHQ ineq} can also be obtained via an application of Lemma~\ref{LLAL}. However, it is worth noting that there exist cases where constructing the required polynomials yields a much slicker and simpler proof than constructing the vectors required to apply Lemma~\ref{LLAL}; see \cite{HHQ} and \cite{MN18}. It is also interesting to note that it has been proved by Terekhov and Zhukovskii~\cite{TZ} that the method of~\cite{BBMR} cannot always achieve tight lower bounds. Terekhov and Zhukovskii~\cite{TZ} also introduce an improved modification of Lemma~\ref{LLAL} which is able to find tight bounds in cases where Lemma~\ref{LLAL} fails to. It would be interesting to investigate the potential limitations of this new method.

\bibliographystyle{abbrv} 
	\bibliography{References}

\pagebreak

\appendix

\section{Proof of \thref{Gen Stars exact}}

The goal of this section is to provide a proof of \thref{Gen Stars exact}, restated here for convenience.

\GenStarsExact*

As noted in Section \ref{Subsec:Stars}, \thref{Gen Stars exact} follows from \thref{stars upper,stars lower}, which we prove below, together with \thref{HHQ ineq}. In what follows, let the vertices of $S_k$ be labelled $0,1,...,k$, where $0$ is the centre of the star. Let $G_0,G_1,...,G_k$ denote the $k+1$ copies of $G$ in $G\square S_k$ corresponding to the $k+1$ vertices in $S_k$. Here, $G_0$ corresponds to the centre of $S_k$. For each vertex $v\in V(G)$, denote its corresponding vertex in $G_i$ by $v_i$, where $0\leq i\leq k$. 

\begin{proof}[Proof of Proposition~\ref{stars upper}]
If $\delta(G)\geq r$, we construct an $r$-percolating set $F\subseteq E(G\square S_k)$ for $G\square S_k$ as follows: For each $i\in [k]$, pick an optimal $(r-1)$-percolating set $F_i$ on $G_i$. Finally, let $F_0$ be an optimal $r$-percolating set for $G_0$. Let $F=\bigcup_{i=0}^{k} F_i$. Note that $F$ will infect all edges in $G\square S_k$, starting with those in $G_0$ and then spreading out to those in $G_1,...,G_k$. Indeed, by our choice of $F_0$, after running the $r$-bond bootstrap percolation process on $G_0$, all edges in $G_0$ will be infected. Since each vertex in $G$ has degree at least $r$, we can then infect all edges between $G_0$ and $G_i$, for each $i\in [k]$. Now, for $i\in [k]$, each vertex $v_i\in V(G_i)$ has an infected incident edge coming from $G_0$. This, together with $F_i$, will infect all edges in $G_i$. Therefore, the set $F$ percolates in $G\square S_k$, as required.

It remains to consider the case where $G$ contains vertices of degree less than $r$. Let the set $F$ be as define above. We add the following edges to $F$: 
\begin{itemize}[nosep]
\item For each $t\in [k-1]$ and each $v\in V(G)$ with $\deg_G(v)=r-t$, add the edges $v_0v_i$ for $i\in [t]$ to the set $F$. Note that adding these edges will guarantee that, once $G_0$ is fully infected, the infection will spread to the edges $v_0v_j$ where $t+1\leq j\leq k$, and we can then continue as above. This adds $\sum_{t=1}^{k-1} t d^G_{r-t}$ edges to the set $F$. 
\item For each $v\in V(G)$ with $\deg_G(v)\leq r-k$, add all edges $v_0v_i$ for $i\in [k]$ to $F$. This adds $k\sum_{t=k}^r d^G_{r-t}$ edges to the set $F$. 
\end{itemize}
Altogether, we added $\sum_{t=1}^{k-1} t d^G_{r-t} + k\sum_{t=k}^r d^G_{r-t}$ edges to the $m_e(G,r) + km_e(G,r-1)$ edges initially in $F$. Therefore, since the resulting set $F$ percolates in $G$, we have established the desired upper bound on $m_e(G\square S_k,r)$.
\end{proof}

\begin{proof}[Proof of Proposition~\ref{stars lower}]
Let $\alpha_1,...,\alpha_k$ be distinct real numbers not in $c(E)$. Define $c^{\prime}$ to be the proper edge-colouring of $G\square S_k$ that is consistent with $c$ on each of $G_0,G_1,...,G_k$, and where $c^{\prime}(v_0v_i)=\alpha_i$ for all $i\in [k]$ and all $v\in V(G)$. 

Let 
$N:=\dim\left(W^r_{G,c}\right) + k \dim\left(W^{r-1}_{G,c}\right) + \sum_{t=1}^{k-1}t d_{r-t} + k \sum_{t=k}^r d_{r-t}$. 
In order to prove \thref{stars lower}, we will find a set $Y$ of $N$ linearly independent vectors in $W^r_{G\square S_k,c^\prime}$. We begin by finding a set $A$ of $\dim\left(W^r_{G,c}\right) + k \dim\left(W^{r-1}_{G,c}\right)$ linearly independent vectors in $W^r_{G\square S_k,c^\prime}$.

First, consider a basis $B^{(r)}$ for $W^r_{G,c}$. For each $\q=(q_v)_{v\in V(G)}$ in $B^{(r)}$, define the vector $\p_\q^{0}=(p_u)_{u\in V(G\square S_k)}$ so that $p_{v_i}=q_v$ for all $0\leq i\leq k$ and all $v\in V(G)$. Trivially, the two conditions in \thref{W} are satisfied, and thus $\p_\q^{0}\in W^r_{G\square S_k,c^{\prime}}$. Let $A_0:=\{\p_\q^{0}:q\in B^{(r)}\}$. Note that the restriction of a vector $\p_\q^{0}\in A_0$ to $G_0$ equals $\q$, and thus the vectors in $A_0$ are linearly independent since $B^{(r)}$ is a basis.

Now, consider a basis $B^{(r-1)}$ for $W^{r-1}_{G,c}$. Fix $\ell\in [k]$. For each $\q=(q_v)_{v\in V(G)}$ in $B^{(r-1)}$, define the vector $\p_\q^{\ell}=(p_u)_{u\in V(G\square S_k)}$ so that, for each vertex $v\in V(G)$, we have $p_{v_\ell}=q_v(x-\alpha_\ell)$, and $p_{v_i}\equiv 0$ for all other $0\leq i\leq k$. For each $v\in V(G)$, since $\q\in W^{r-1}_{G,c}$ we have
\begin{align*}
    \deg(p_{v_\ell}) &= \deg(q_v) + 1 \leq \min\{r-1,\deg_G(v)\}= \min\{r,\deg(v_\ell)\}-1.
\end{align*}
\noindent Moreover, $\p_{v_\ell}(\alpha_\ell)=0=\p_{v_0}(\alpha_\ell)$. Since $p_u(c'_{uv})=p_v(c'_{uv})$ for all other edges $uv\in E(G\square S_k)$, the conditions in \thref{W} are satisfied, and thus these vectors $\p_\q^{\ell}$ belong to $W^{r}_{G\square S_k,c^\prime}$. 
For each $\ell\in [k]$, define $A_\ell=\{\p_\q^{\ell}:\q\in B^{(r-1)}\}$. Note that, for each $\ell\in [k]$, the restriction of $\p_\q^{\ell}$ to $G_\ell$ is the vector $(x-\alpha_\ell)\q$. Thus, since $x-\alpha_\ell\not\equiv 0$ and $B^{(r-1)}$ is a basis, the vectors in $A_\ell$ are linearly independent. Define $A:=\bigcup_{i=0}^k A_i$. Note that $|A_0|=\dim\left(W^r_{G,c}\right)$, and $|A_i|=\dim\left(W^{r-1}_{G,c}\right)$ for each $i\in [k]$. Hence $|A|=\dim\left(W^r_{G,c}\right) + k \dim\left(W^{r-1}_{G,c}\right)$. 

\begin{claim}
\thlabel{claim:stars A ind}
    The vectors in $A$ are linearly independent.
\end{claim}

\noindent\emph{Proof of \thref{claim:stars A ind}:} Suppose for a contradiction that there exists a linear combination $S=\sum_{\p\in A} \beta_\p \p = 0$, where $\beta_\p\neq0$ for some $\p\in A$. Write $S=S_0+S_1+\cdots+S_k$ where $S_i=\sum_{\p\in A_i}\beta_\p \p$ for each $0\leq i\leq k$. Let $0\leq i\leq k$ be the smallest index such that there exists $\p\in A_i$ with $\beta_\p\neq 0$. If $i=0$, then $p_{v_0}\equiv 0$ for all $\p\in A_1\cup\cdots\cup A_k$ and $v\in V(G)$. Hence the restriction of $S$ to the coordinates corresponding to vertices in $G_0$ is equal to the restriction of $S_0$ to $G_0$. However, as the vectors in $A_0$ are linearly independent, this contradicts $\beta_\p\neq 0$ for some $\p\in A_0$. Hence $i\neq 0$, and thus $\beta_\p=0$ for all $\p\in A_0$. 

Otherwise, if $i\in [k]$, then $p_{v_i}\equiv 0$ for each $\p\in \bigcup_{j\in [k]\setminus\{i\}} A_j$ and $v\in V(G)$. Therefore, since $\beta_\p=0$ for all $\p\in A_0$, the restriction of the sum $S$ to the coordinates corresponding to vertices in $G_i$ is equal to the restriction of $S_i$ to $G_i$. Again, since the vectors in $A_i$ are linearly independent, this contradicts our assumption that $\beta_\p\neq 0$ for some $\p\in A_i$.~$_{(\square)}$ 

Now, set $Y=A$. To achieve the desired lower bound on $\dim\left(W^r_{G\square S_k,c^\prime}\right)$, we must extend the set $Y$ to include an additional $\sum_{t=1}^{k-1}t d^G_{r-t} + k \sum_{t=k}^r d^G_{r-t}$ vectors, all of which are linearly independent. To this end, consider the set of vectors $Z_c\in \R^{|V(G)|}$ as defined in \thref{Z}. 

\begin{claim}
\thlabel{stars claim}
For each non-zero $\p\in \vspan(A)$, there exists $\z\in Z_c$ such that $\p(\z)\neq 0$.
\end{claim}

\noindent\emph{Proof of Claim:} Let $\p= (p_u)_{u\in V(G\square S_k)}$ be a non-zero vector in $\vspan(A)$. Let $j$ be the smallest index for which there exists some $v_j\in V(G_j)$ with $p_{v_j}\not\equiv 0$. If $j=0$, then by definition of the vectors in $A$, we can write $p_{v_0}= q_v$ for some $\q=(q_u)_{u\in V(G)}$ in $ \vspan(B^{(r)})=W^{r}_{G,c}$. Otherwise, if $j\in [k]$, then by definition of the vectors in $A$ and the minimality of $j$, we can write $p_{v_j}= q_v(x-\alpha_j)$ for some $\q=(q_u)_{u\in V(G)}$ in $ \vspan(B^{(r-1)})=W^{r-1}_{G,c}$. In any case, $\deg(q_v)\leq \deg_G(v)-1$ by \thref{W}. Therefore, as the colours in $\{\alpha_1,...,\alpha_k\}$ were chosen to be distinct from those in $c(E)$, by \thref{zeros lemma} there exists $\z\in Z_c$ such that $\p(\z)\neq 0$.$_{(\square)}$

Thus, to finish the proof, it suffices to find a set $X$ of $\sum_{t=1}^{k-1}t d^G_{r-t} + k \sum_{t=k}^r d^G_{r-t}$ linearly independent vectors in $W^r_{G\square S_k, c^\prime}$ that evaluate to zero on $Z_c$. This, together with \thref{stars claim}, will guarantee that these vectors are independent from those in $A$. 

First, for each $v\in V(G)$ with $\deg_G(v)\leq r-k$ and each $\ell\in [k]$, define the vector $\p_v^\ell = (p_u)_{u\in V(G\square S_k)}$ as follows: 
\begin{itemize}[nosep]
    \item Let $p_{v_0}$ be the unique polynomial of degree $\deg(v_0)-1$ that is equal to $1$ at $c'(v_0v_\ell)$ and equal to zero on all $\deg(v_0)-1$ other colours of the edges incident with $v_0$ in $G\square S_k$. That is, $p_{v_0}=\beta_0\prod_{u\in N(v_0)\setminus\{v_\ell\}}(x-c'_{v_0u})$, where $\beta_0$ is a non-zero constant chosen so that $p_{v_0}(c'_{v_0v_\ell})=1$. 
    \item Let $p_{v_\ell}$ be the unique polynomial of degree $\deg(v_\ell)-1$ that is equal to $1$ at $c'(v_0v_\ell)$ and equal to zero on all $\deg(v_\ell)-1$ other colours of the edges incident with $v_\ell$ in $G\square S_k$. That is, $p_{v_\ell}=\beta_\ell\prod_{u\in N(v_\ell)\setminus\{v_0\}}(x-c'_{v_\ell u})$, where $\beta_\ell$ is a non-zero constant chosen so that $p_{v_\ell}(c'_{v_0v_\ell})=1$.
    \item Let $p_u\equiv 0$ for all $u\in V(G\square S_k)\setminus \{v_0,v_\ell\}$.
\end{itemize}   
Since $\deg_G(v)\leq r-k$, we have $\deg(p_{v_0}) = \deg(v_0)-1 = \deg_G(v) + k-1\leq r-1$. 
Moreover, $k\geq 1$ implies that $\deg(p_{v_\ell}) = \deg(v_\ell)-1 = \deg_G(v)\leq r-1$. 
Since $p_u(c'_{uv})=p_v(c'_{uv})$ for all $uv\in E(G\square S_k)$, the conditions in \thref{W} are satisfied, and thus $\p_v^\ell\in W^r_{G\square S_k,c^\prime}$. Moreover, $\p_v^\ell$ evaluates to $0$ on all of $Z_c$. Let $X_0=\{\p_v^\ell: \deg_G(v)\leq r-k \text{ and }\ell\in[k]\}$. Note that $|X_0|=k\sum_{t=k}^r d_{r-t}^G$.

Now, fix $t\in [k-1]$. For each $\ell\in [t]$ and each $v\in V(G)$ with $\deg_G(v)= r-t$, define $\p_v^\ell = (p_u)_{u\in V(G\square S_k)}$ as follows: Let $p_u\equiv 0$ for all $u\notin \{v_0,v_\ell\}\cup\{v_{t+1},...,v_k\}$. Let
\[p_{v_0} = \prod\limits_{\substack{i\in [t]\\i\neq \ell}}(x-\alpha_i) \prod\limits_{u\in N_G(v)}(x-c_{uv}) \ \ \ \text{and}\ \ \ p_{v_j} = \gamma_j \prod\limits_{u\in N_G(v)}(x-c_{uv})\]
\noindent for all $j\in \{\ell\}\cup\{t+1,...,k\}$, where $\gamma_j$ is a constant chosen so that $p_{v_j}(\alpha_j)=p_{v_0}(\alpha_j)$. Note that $\gamma_j\neq 0$ since $p_{v_0}(\alpha_j)\neq 0$. Now, $\deg_G(v)=r-t$ implies $\deg(p_{v_0}) = \deg_G(v) + t-1 = r-1$. Moreover,
$\deg(p_{v_0}) = \deg_G(v) + t-1 = \deg(v_0) - k+t-1< \deg(v_0)-1$ as $t\leq k-1$.
Hence $\deg(p_{v_0})\leq \min\{r,\deg(v_0)\}-1$. Furthermore, for $j\in \{\ell\}\cup\{t+1,...,k\}$, we have $\deg(p_{v_j}) = \deg_G(v) \leq \min\{r-1,\deg(v_j)-1\} =\min\{r,\deg(v_j)\}-1$ as $\deg_G(v)=r-t\leq r-1$ and $\deg_G(v)=\deg(v_j)-1$.
Since $p_u(c'_{uv})=p_v(c'_{uv})$ for all $uv\in E(G\square S_k)$, the conditions in \thref{W} are satisfied, and thus $\p_v^\ell\in W^r_{G\square S_k,c^\prime}$. Moreover, $\p_v^\ell$ evaluates to $0$ on all of $Z_c$. For each $t\in [k-1]$, let $X_t = \{\p_v^\ell: \deg_G(v)= r-t \text{ and }\ell\in[t]\}$. Note that $|X_t|=td_{r-t}^G$. 

Define $X:=\bigcup_{i=0}^{k-1} X_i$. Hence $|X|=\sum_{i=0}^{k-1} |X_i| = k \sum_{t=k}^r d^G_{r-t}+ \sum_{t=1}^{k-1}t d^G_{r-t}$.  

\begin{claim}
\thlabel{claim:stars X ind}
    The vectors in $X$ are linearly independent.
\end{claim}

\noindent\emph{Proof of \thref{claim:stars X ind}:} Suppose for a contradiction that there is a linear combination $S=\sum_{\p\in X} \beta_\p \p = 0$, where $\beta_\p\neq0$ for some $\p\in X$. Note that, for each $v\in V(G)$ with $\deg_G(v)\leq r-k$ and each $\ell\in [k]$, we have $p_{v_\ell}\not\equiv 0$ for the vector $\p_v^\ell\in X_0\subseteq X$, but $p_{v_\ell}\equiv 0$ for all other vectors $\p\in X$. Similarly, fixing $t\in [k-1]$, for each $v\in V(G)$ with $\deg_G(v)=r-t$ and each $\ell\in [t]$, we have $p_{v_\ell}\not\equiv 0$ for the vector $\p_v^\ell\in X_t\subseteq X$, but $p_{v_\ell}\equiv 0$ for all other vectors $\p\in X$. Therefore, for each vector $\p\in X$, there exists a coordinate $p_u$ which is non-zero in $\p$ but identically zero in all other vectors in $X$. Thus, restricting the sum $S$ to the entry corresponding to $p_u$, we see that $\beta_\p=0$ for all $\p\in X$, a contradiction.~$_{(\square)}$ 

Let $Y=A\cup X$. Therefore, by \thref{claim:stars A ind,stars claim,claim:stars X ind}, we have found 
\[|Y|=|A|+|X|=\dim\left(W^r_{G,c}\right) + k \dim\left(W^{r-1}_{G,c}\right) + \sum_{t=1}^{k-1}t d^G_{r-t} + k \sum_{t=k}^r d^G_{r-t}\] linearly independent vectors in $W^r_{G\square S_k,c^\prime}$, as required.
\end{proof}

\end{document}